\documentclass[10pt]{article}
\usepackage{amsfonts}
\usepackage{amssymb}
\usepackage{amsthm}
\usepackage{amsmath}
\usepackage{bm}
\usepackage{amscd} 

\usepackage{hyperref}

\numberwithin{equation}{section} \DeclareMathSizes{2}{10}{12}{13}
\newtheorem{thm}{Proposition}[section]

\newtheorem{Thm}[thm]{Theorem}

\newtheorem{lem}[thm]{Lemma}
\newtheorem{defn}[thm]{Definition}

\title{
Quasimodular Hecke algebras and Hopf actions}
\author{Abhishek Banerjee
}
\date{ }

\begin{document} 

\maketitle

\medskip
\centerline{\emph{Dept. of Mathematics, Indian Institute of Science, Bangalore, Karnataka - 560012, India.}}
\centerline{\emph{Email: abhishekbanerjee1313@gmail.com}}

\begin{abstract}
Let $\Gamma=\Gamma(N)$ be a principal congruence subgroup of $SL_2(\mathbb Z)$. In this paper, we extend
the theory of modular Hecke algebras due to Connes and Moscovici to define the algebra $\mathcal Q(\Gamma)$ of quasimodular Hecke
operators of level $\Gamma$. Then, $\mathcal Q(\Gamma)$ carries an action of ``the Hopf algebra $\mathcal H_1$ of codimension
$1$ foliations'' that also acts on the modular Hecke algebra $\mathcal A(\Gamma)$ of Connes and Moscovici. However, 
in the case of quasimodular forms, we have several new operators  acting on the quasimodular Hecke algebra
$\mathcal Q(\Gamma)$. Further, for each $\sigma\in SL_2(\mathbb Z)$, we
introduce the collection $\mathcal Q_\sigma(\Gamma)$ of quasimodular Hecke operators of level $\Gamma$ twisted
by $\sigma$. Then, $\mathcal Q_\sigma(\Gamma)$ is a right $\mathcal Q(\Gamma)$-module and is endowed with 
a pairing $(\_\_,\_\_):\mathcal Q_\sigma(\Gamma)\otimes \mathcal Q_\sigma(\Gamma)\longrightarrow \mathcal Q_\sigma(\Gamma)$. 
We show that there is a ``Hopf action'' of a certain Hopf algebra $\mathfrak{h}_1$ on the pairing on $\mathcal Q_\sigma(\Gamma)$. Finally,
for any $\sigma\in SL_2(\mathbb Z)$, we consider operators
acting between the levels of the graded module   $\mathbb Q_\sigma(\Gamma)=\underset{m\in \mathbb Z}{\oplus}\mathcal Q_{\sigma(m)}(\Gamma)$, where
$\sigma(m)=\begin{pmatrix} 1 & m \\ 0 & 1 \\ \end{pmatrix}\cdot \sigma$ for any $m\in \mathbb Z$. The pairing on 
$\mathcal Q_\sigma(\Gamma)$ can be extended to a graded pairing on $\mathbb Q_\sigma(\Gamma)$ and we show
that there is a Hopf action of a larger Hopf algebra $\mathfrak{h}_{\mathbb Z}\supseteq \mathfrak{h}_1$ on the pairing
on $\mathbb Q_\sigma(\Gamma)$. 
\end{abstract}


\medskip

\medskip

\section{Introduction}

\medskip

\medskip
Let $N\geq 1$ be an integer and let  $\Gamma=\Gamma(N)$ be a principal congruence subgroup of
$SL_2(\mathbb Z)$.  In \cite{CM1}, \cite{CM2}, Connes and Moscovici have introduced 
the ``modular Hecke algebra'' $\mathcal A(\Gamma)$ that combines the pointwise product on
modular forms with the action of Hecke operators. Further, Connes and Moscovici have shown that
the modular Hecke algebra $\mathcal A(\Gamma)$ carries an action of ``the Hopf algebra
$\mathcal H_1$ of codimension $1$ foliations''. The Hopf algebra $\mathcal H_1$ is part of a larger
family of Hopf algebras $\{\mathcal H_n|n\geq 1\}$ defined in \cite{CM0}. The objective of this paper
is to introduce and study quasimodular Hecke algebras $\mathcal Q(\Gamma)$ that similarly combine
the pointwise product on quasimodular forms with the action of Hecke operators. We will see that the quasimodular
Hecke algebra $\mathcal Q(\Gamma)$ carries several other operators in addition to an action of 
$\mathcal H_1$. Further, we will also
study the collection $\mathcal Q_\sigma(\Gamma)$ of quasimodular Hecke operators twisted
by some $\sigma\in SL_2(\mathbb Z)$. The latter is a generalization of our theory of twisted modular Hecke
operators introduced in \cite{AB1}.  

\medskip
We now describe the paper in detail. In Section 2, we briefly recall the notion of modular Hecke algebras
of Connes and Moscovici \cite{CM1}, \cite{CM2}. We let $\mathcal{QM}$ be the ``quasimodular tower'', i.e., 
$\mathcal{QM}$ is the colimit
over all $N$ of the spaces $\mathcal{QM}(\Gamma(N))$ of quasimodular forms of level $\Gamma(N)$ (see \eqref{2.8}). We define a quasimodular Hecke operator
of level $\Gamma$ to be a function of finite support from $\Gamma\backslash GL_2^+(\mathbb Q)$
to the quasimodular tower $\mathcal{QM}$ satisfying a certain covariance condition (see Definition \ref{maindef}). 
We then show that   the collection $\mathcal Q(\Gamma)$ of quasimodular Hecke operators
of level $\Gamma$ carries an algebra structure $(\mathcal Q(\Gamma),\ast)$ by considering a convolution product over cosets of $\Gamma$ in $GL_2^+(\mathbb Q)$. 
Further, the modular Hecke algebra of Connes and Moscovici embeds naturally as a subalgebra of $
\mathcal Q(\Gamma)$. We also show that the quasimodular Hecke operators of level $\Gamma$ act on 
quasimodular forms of level $\Gamma$, i.e., $\mathcal{QM}(\Gamma)$ is a left $\mathcal Q(\Gamma)$-module. 
In this section, we will also define a second algebra structure $(\mathcal Q(\Gamma),\ast^r)$ on $\mathcal Q(\Gamma)$ by considering the 
convolution product over cosets of $\Gamma$ in $SL_2(\mathbb Z)$. When we consider $\mathcal Q(\Gamma)$
as an algebra equipped with this latter product $\ast^r$, it will be denoted by $\mathcal Q^r(\Gamma)=(\mathcal Q(\Gamma),\ast^r)$.

\medskip
In Section 3, we define Lie algebra and  Hopf algebra actions on $\mathcal Q(\Gamma)$. Given a quasimodular form
$f\in \mathcal{QM}(\Gamma)$ of level $\Gamma$, it is well known that we can write $f$ as a sum
\begin{equation}\label{1.0ner}
f=\sum_{i=0}^sa_i(f)\cdot G_2^i
\end{equation}  where the coefficients $a_i(f)$ are modular forms of level $\Gamma$ and $G_2$ is the classical
Eisenstein series of weight $2$. Therefore, we can consider two different sets of operators on the quasimodular tower 
$\mathcal{QM}$: those which
act on the powers of $G_2$ appearing in the expression for $f$ and those which act on the modular coefficients $a_i(f)$. The collection
of operators acting on the modular coefficients $a_i(f)$ are studied in Section 3.2. These induce on $\mathcal Q(\Gamma)$ analogues of operators acting on
the modular Hecke algebra $\mathcal A(\Gamma)$ of Connes and Moscovici and we show that
$\mathcal Q(\Gamma)$ carries an action of the same Hopf algebra $\mathcal H_1$ of codimension $1$ foliations
that acts on $\mathcal A(\Gamma)$.  On the other hand, by considering operators on $\mathcal{QM}$ that act on the powers
of $G_2$ appearing in \eqref{1.0ner},  we are able to define  additional operators $D$, $\{T_k^l\}_{k\geq 1,l\geq 0}$ and 
$\{\phi^{(m)}\}_{m\geq 1}$ on $\mathcal Q(\Gamma)$ (see Section 3.1). Further, we show that these operators satisfy the following commutator
relations: 
\begin{equation}\label{1.2ner}
\begin{array}{c}
[T_k^l,T_{k'}^{l'}]=(k'-k)T_{k+k'-2}^{l+l'}\\
\mbox{$[D,\phi^{(m)}]=0 \qquad [T_k^l,\phi^{(m)}]=0 \qquad [\phi^{(m)},
\phi^{(m')}]=0$} \\
\mbox{$[T_k^l,D]=\frac{5}{24}(k-1)T^{l+1}_{k-1}-\frac{1}{2}(k-3)T^l_{k+1}$}\\
\end{array}
\end{equation} We then consider the  Lie algebra $\mathcal L$ generated by the symbols 
$D$, $\{T_k^l\}_{k\geq 1,l\geq 0}$,  
$\{\phi^{(m)}\}_{m\geq 1}$ satisfying the  commutator relations in \eqref{1.2ner}. Then, there
is a Lie action of $\mathcal L$ on $\mathcal Q(\Gamma)$. Finally, let $\mathcal H$ be the Hopf algebra given by the
universal enveloping algebra $\mathcal U(\mathcal L)$ of $\mathcal L$.  Then, we show that $\mathcal H$ has a Hopf action with
respect to the product $\ast^r$ on $\mathcal Q(\Gamma)$  and this action captures the operators $D$, $\{T_k^l\}_{k\geq 1,l\geq 0}$ and 
$\{\phi^{(m)}\}_{m\geq 1}$ on $\mathcal Q(\Gamma)$. In other words, $\mathcal H$ acts on $\mathcal Q(\Gamma)$
such that:
\begin{equation}\label{1.3ner}
h(F^1\ast^r F^2)=\sum h_{(1)}(F^1)\ast^r h_{(2)}(F^2) \qquad\forall \textrm{ } h\in \mathcal H, \textrm{ }F^1,F^2\in \mathcal Q(\Gamma)
\end{equation} where the coproduct $\Delta:\mathcal H\longrightarrow \mathcal H\otimes \mathcal H$ is given by
$\Delta(h)=\sum h_{(1)}\otimes h_{(2)}$ for any $h\in \mathcal H$. 

\medskip
In Section 4, we develop the theory of twisted quasimodular Hecke operators. For any 
$\sigma\in SL_2(\mathbb Z)$, we define in Section 4.1 the collection $\mathcal Q_\sigma(\Gamma)$
of quasimodular Hecke operators of level $\Gamma$ twisted by $\sigma$. When $\sigma=1$, this reduces
to the original definition of $\mathcal Q(\Gamma)$. In general, $\mathcal Q_\sigma(\Gamma)$ is not
an algebra but we show that $\mathcal Q_\sigma(\Gamma)$ carries a pairing:
\begin{equation}\label{1.7m}
(\_\_,\_\_):\mathcal Q_\sigma(\Gamma)\otimes \mathcal Q_\sigma(\Gamma)\longrightarrow \mathcal Q_\sigma(\Gamma)
\end{equation} Further, we show that $\mathcal Q_\sigma(\Gamma)$ may be equipped with the structure
of a right $\mathcal Q(\Gamma)$-module. We can also extend the action of the Hopf algebra 
$\mathcal H_1$ of codimension $1$ foliations to $\mathcal Q_\sigma(\Gamma)$. In fact, we show that
$\mathcal H_1$ has a ``Hopf action'' on the right $\mathcal Q(\Gamma)$ module 
$\mathcal Q_\sigma(\Gamma)$, i.e.,
\begin{equation}
h(F^1\ast F^2)=\sum h_{(1)}(F^1)\ast h_{(2)}(F^2) \qquad\forall \textrm{ } h\in \mathcal H_1, \textrm{ }F^1\in \mathcal Q_\sigma(\Gamma), \textrm{ }
F^2\in \mathcal Q(\Gamma)
\end{equation} where the coproduct $\Delta:\mathcal H_1\longrightarrow \mathcal H_1\otimes \mathcal H_1$ is given by
$\Delta(h)=\sum h_{(1)}\otimes h_{(2)}$ for any $h\in \mathcal H_1$. We recall from \cite{CM1} that 
$\mathcal H_1$ is equal as an algebra to the universal enveloping algebra of the Lie algebra $\mathcal L_1$ with
generators $X$, $Y$, $\{\delta_n\}_{n\geq 1}$ satisfying the following relations:
\begin{equation}
[Y,X]=X \quad [X,\delta_n]=\delta_{n+1}\quad [Y,\delta_n]=n\delta_n\quad [\delta_k,\delta_l]=0\qquad\forall\textrm{ }k,l,n\geq 1
\end{equation} Then, we can consider the smaller Lie algebra $\mathfrak{l}_1\subseteq \mathcal L_1$ with two generators
$X$, $Y$ satisfying $[Y,X]=X$. If we let $\mathfrak{h}_1$ be the Hopf algebra that is  the universal enveloping algebra 
of $\mathfrak{l}_1$, we show that the pairing in \eqref{1.7m} on $\mathcal Q_\sigma(\Gamma)$ carries a
``Hopf action'' of $\mathfrak{h}_1$. In other words, we have:
\begin{equation}\label{1.10m}
h(F^1,F^2)=\sum (h_{(1)}(F^1),h_{(2)}(F^2))\qquad\forall \textrm{ } h\in \mathfrak{h}_1, \textrm{ }F^1,F^2\in \mathcal Q_\sigma(\Gamma)
\end{equation} where the coproduct $\Delta:\mathfrak{h}_1\longrightarrow \mathfrak{h}_1\otimes \mathfrak{h}_1$ is given by
$\Delta(h)=\sum h_{(1)}\otimes h_{(2)}$ for any $h\in \mathfrak{h}_1$. In Section 4.2, we consider operators
between the modules $\mathcal Q_\sigma(\Gamma)$ as $\sigma$ varies over $SL_2(\mathbb Z)$. More precisely, for any
$\tau, \sigma\in SL_2(\mathbb Z)$, we define a morphism:
\begin{equation}
X_\tau:\mathcal Q_\sigma(\Gamma)\longrightarrow \mathcal Q_{\tau\sigma}(\Gamma)
\end{equation} In particular, this gives us operators acting between the levels of the graded module
\begin{equation}
\mathbb Q_\sigma(\Gamma)=\bigoplus_{m\in \mathbb Z}\mathcal Q_{\sigma(m)}(\Gamma)
\end{equation} where for any $\sigma \in SL_2(\mathbb Z)$, we set $\sigma(m)
=\begin{pmatrix} 1 & m \\ 0 & 1 \\ \end{pmatrix}\cdot \sigma$. Further, we generalize the  pairing 
on $\mathcal Q_\sigma(\Gamma)$ in \eqref{1.7m} to a pairing:
\begin{equation}\label{1.13m}
(\_\_,\_\_):\mathcal Q_{\tau_1\sigma}(\Gamma)\otimes \mathcal Q_{\tau_2\sigma}(\Gamma)\longrightarrow \mathcal Q_{\tau_1\tau_2\sigma}(\Gamma)
\end{equation} where $\tau_1$, $\tau_2$  are commuting matrices in $SL_2(\mathbb Z)$. In particular, \eqref{1.13m} gives us a pairing
$\mathcal Q_{\sigma(m)}(\Gamma)\otimes \mathcal Q_{\sigma(n)}(\Gamma)\longrightarrow \mathcal Q_{\sigma(m+n)}(\Gamma)$, $\forall$
$m$, $n\in \mathbb Z$ and hence a  pairing on the tower $\mathbb Q_\sigma(\Gamma)$. Finally, we consider 
the Lie algebra $\mathfrak{l}_{\mathbb Z}\supseteq \mathfrak{l}_1$ with generators $\{Z,X_n|n\in \mathbb Z\}$ satisfying the following
commutator relations:
\begin{equation}
[Z,X_n]=(n+1)X_n\qquad [X_n,X_{n'}]=0\qquad \forall\textrm{ }n,n'\in \mathbb Z
\end{equation} Then, if we let $\mathfrak{h}_{\mathbb Z}$ be the Hopf algebra that is the universal enveloping
algebra of $\mathfrak{l}_{\mathbb Z}$, we show that $\mathfrak{h}_{\mathbb Z}$ has a Hopf action on the 
pairing on $\mathbb Q_\sigma(\Gamma)$. In other words, for any $F^1$, $F^2\in \mathbb Q_{\sigma}(\Gamma)$, we have
\begin{equation}\label{last4.48}
h(F^1,F^2)=\sum (h_{(1)}(F^1),h_{(2)}(F^2))\qquad \forall\textrm{ }h\in \mathfrak{h}_{\mathbb Z}
\end{equation} where the coproduct $\Delta:\mathfrak h_{\mathbb Z}
\longrightarrow \mathfrak h_{\mathbb Z}\otimes \mathfrak h_{\mathbb Z}$ is defined
by setting $\Delta(h):=\sum h_{(1)}\otimes h_{(2)}$ for each $h\in \mathfrak h_{\mathbb Z}$.

\medskip

\medskip
\section{The Quasimodular Hecke algebra}

\medskip

\medskip
We begin this section by briefly recalling the notion of quasimodular forms. The notion of quasimodular forms is due to Kaneko and Zagier \cite{KZ}. The theory has been further developed in Zagier \cite{Zag}. For an introduction to the basic theory of quasimodular forms, we refer the reader to the exposition of Royer \cite{Royer}. 

\medskip
Throughout, let $\mathbb H\subseteq \mathbb C$ be the upper half plane. Then, there is a well known
action of $SL_2(\mathbb Z)$ on $\mathbb H$:
\begin{equation}
z\mapsto \frac{az+b}{cz+d}\qquad \forall \textrm{ }z\in \mathbb H, \begin{pmatrix} a & b \\ 
c & d \\ \end{pmatrix} \in SL_2(\mathbb Z)
\end{equation} For any $N\geq 1$, we denote by $\Gamma(N)$ the following principal congruence
subgroup of $SL_2(\mathbb Z)$:
\begin{equation}
\Gamma(N):=\left\{ \left. \begin{pmatrix} a & b \\ 
c & d \\ \end{pmatrix}\in SL_2(\mathbb Z) \right |  \begin{pmatrix} a & b \\ 
c & d \\ \end{pmatrix} \equiv \begin{pmatrix} 1 & 0 \\ 
0 & 1 \\ \end{pmatrix} \mbox{($mod$ $N$)} \right\}
\end{equation} In particular, $\Gamma(1)=SL_2(\mathbb Z)$. We are now ready to define
quasimodular forms. 

\medskip
\begin{defn}\label{Def2.1} Let $f:\mathbb H\longrightarrow \mathbb C$ be a holomorphic function and let
$N\geq 1$, $k$, $s\geq 0$ be integers. Then, the function $f$ is a quasimodular form of 
level $N$, weight $k$ and depth $s$ if there exist holomorphic functions $f_0$, $f_1$, ..., $f_s:
\mathbb H\longrightarrow \mathbb C$ with $f_s\ne 0$ such that:
\begin{equation}\label{2.3}
(cz+d)^{-k}f\left(\frac{az+b}{cz+d}\right)=\sum_{j=0}^s f_j(z)\left(\frac{c}{cz+d}\right)^j
\end{equation} for any matrix $\begin{pmatrix} a & b \\ 
c & d \\ \end{pmatrix} \in \Gamma(N)$. The collection of quasimodular forms of level $N$, weight $k$ and depth $s$ will be denoted by $\mathcal{QM}_k^s(\Gamma(N))$. By convention, we let the zero function $0\in \mathcal{QM}_k^0(\Gamma(N))$ for every $k\geq 0$, $N\geq 1$. 

\end{defn} 

\medskip
More generally, for any holomorphic function $f:\mathbb H\longrightarrow \mathbb C$ and any matrix
$\alpha = \begin{pmatrix} a & b \\ 
c & d \\ \end{pmatrix}\in GL_2^+(\mathbb Q)$, we define:
\begin{equation}
(f\vert_k\alpha)(z):=(cz+d)^{-k}f\left(\frac{az+b}{cz+d}\right) \qquad \forall \textrm{ }k\geq 0
\end{equation} Then, we can say that $f$ is quasimodular of level $N$, weight $k$ and depth $s$ 
if there exist holomorphic functions $f_0$, $f_1$, ..., $f_s:
\mathbb H\longrightarrow \mathbb C$ with $f_s\ne 0$ such that:
\begin{equation}
(f\vert_k\gamma)(z) = \sum_{j=0}^s f_j(z)\left(\frac{c}{cz+d}\right)^j \qquad\forall\textrm{ } 
\gamma = \begin{pmatrix} a & b \\ 
c & d \\ \end{pmatrix} \in \Gamma(N)
\end{equation} When the integer $k$ is clear from context, we write $f\vert_k\alpha$ simply
as $f\vert \alpha$ for any $\alpha\in GL_2^+(\mathbb Q)$. Also, it is clear that we have a product:
\begin{equation}
\mathcal{QM}^s_k(\Gamma(N))\otimes \mathcal{QM}^t_l(\Gamma(N))\longrightarrow \mathcal{QM}^{s+t}_{k+l}
(\Gamma(N))
\end{equation} on quasi-modular forms. For any $N\geq 1$, we now define:
\begin{equation}
\mathcal{QM}(\Gamma(N)):=\bigoplus_{s=0}^\infty\bigoplus_{k=0}^\infty 
\mathcal{QM}_k^s(\Gamma(N))
\end{equation} We now consider the direct limit:
\begin{equation}\label{2.8}
\mathcal{QM}:=\underset{N\geq 1}{\varinjlim} \textrm{ }\mathcal{QM}(\Gamma(N))
\end{equation} which we will refer to as the quasimodular tower.  Additionally, for any $k\geq 0$ and  $N\geq 1$, we let $\mathcal M_k(\Gamma(N))$ denote the collection of usual modular forms
of weight $k$ and level $N$. Then, we can define the modular tower $\mathcal M$:
\begin{equation}\label{mtower}
\mathcal{M}:=\underset{N\geq 1}{\varinjlim} \textrm{ }\mathcal{M}(\Gamma(N))\qquad 
\mathcal{M}(\Gamma(N)):=\bigoplus_{k=0}^\infty 
\mathcal{M}_k(\Gamma(N))
\end{equation} We now recall the modular Hecke algebra of Connes and Moscovici \cite{CM1}. 

\medskip
\begin{defn}\label{CMdef} (see \cite[$\S$ 1]{CM1}) Let $\Gamma=\Gamma(N)$ be a principal congruence
subgroup of $SL_2(\mathbb Z)$. A modular Hecke operator of level $\Gamma$ is a function
of finite support 
\begin{equation}
F:\Gamma\backslash GL_2^+(\mathbb Q)\longrightarrow \mathcal{M} \qquad \Gamma\alpha
\mapsto F_\alpha
\end{equation} such that for any $\gamma\in \Gamma$, we have:
\begin{equation}\label{tt2.11}
F_{\alpha\gamma}=F_\alpha|\gamma 
\end{equation} The collection of all modular Hecke operators of level $\Gamma$ will be denoted
by $\mathcal A(\Gamma)$. 
\end{defn}

\medskip 
Our first aim is to define a quasimodular Hecke algebra $\mathcal Q(\Gamma)$ analogous to the modular Hecke algebra $\mathcal A(\Gamma)$ of Connes and Moscovici. For this, we recall the structure
theorem for quasimodular forms, proved by Kaneko and Zagier \cite{KZ}. 

\medskip
\begin{Thm}\label{Th2.1} (see \cite[$\S$ 1, Proposition 1.]{KZ}) Let $\Gamma=\Gamma(N)$ be a principal congruence
subgroup of $SL_2(\mathbb Z)$. For any even number $K\geq 2$, let $G_K$ denote the classical  Eisenstein series of weight $K$: 
\begin{equation}\label{2.12ov}
G_K(z):=-\frac{B_K}{2K}+\sum_{n=1}^\infty\left(\sum_{d|n}d^{K-1}\right)e^{2\pi inz}
\end{equation} where $B_K$ is the $K$-th Bernoulli number and $z\in \mathbb H$. 
Then, every quasimodular form in $\mathcal{QM}(\Gamma)$ can be written uniquely  as a polynomial in $G_2$ with coefficients in 
$\mathcal M(\Gamma)$.  More precisely, for any quasimodular form $f\in \mathcal{QM}^s_k(\Gamma)$, there exist functions $a_0(f)$, $a_1(f)$, ..., $a_s(f)$ such that:
\begin{equation}
f=\underset{i=0}{\overset{s}{\sum}} a_i(f)G_2^i 
\end{equation} where $a_i(f)\in \mathcal M_{k-2i}(\Gamma)$ is a modular form of weight $k-2i$ and level $\Gamma$ for each
$0\leq i\leq s$. 
\end{Thm}

\medskip
We now consider a quasimodular form  $f\in \mathcal{QM}$. For sake of definiteness, we may assume
that $f\in \mathcal{QM}^s_k(\Gamma(N))$, i.e. $f$ is a quasimodular form
of level $N$, weight $k$ and depth $s$. We now define an operation on $\mathcal{QM}$ by setting:
\begin{equation}\label{t2.11}
f||\alpha = \underset{i=0}{\overset{s}{\sum}} (a_i(f)|_{k-2i}\alpha ) G_2^i  \qquad \forall \textrm{ }\alpha\in GL_2^+(\mathbb Q)
\end{equation} where $\{a_i(f)\in \mathcal M_{k-2i}(\Gamma(N))\}_{0\leq i\leq s}$ is the collection of modular forms determining
$f=\sum_{i=0}^s a_i(f)G_2^i$ as in Theorem \ref{Th2.1}. We know  that for any $\alpha
\in GL_2^+(\mathbb Q)$, each $(a_i(f)|_{k-2i}\alpha)$ is an element of the modular tower $\mathcal M$. This shows that $f||\alpha = \underset{i=0}{\overset{s}{\sum}} (a_i(f)|_{k-2i}\alpha ) G_2^i\in \mathcal{QM}$. However, we note that for arbitrary $\alpha\in GL_2^+(\mathbb Q)$ and
$a_i(f)\in \mathcal M_{k-2i}(\Gamma(N))$, it is not necessary that $(a_i(f)|_{k-2i}\alpha )
\in \mathcal M_{k-2i}(\Gamma(N))$. In other words, the operation defined in \eqref{t2.11} on the quasimodular tower $\mathcal{QM}$ does not descend to an endomorphism on each 
$\mathcal{QM}^s_k(\Gamma(N))$.  From the expression in \eqref{t2.11}, it is also clear that:
\begin{equation}\label{tv2.11}
(f\cdot g)||\alpha = (f||\alpha)\cdot (g||\alpha) \qquad f||(\alpha\cdot \beta)=(f||\alpha)||\beta \qquad
\forall\textrm{ }f,g\in \mathcal{QM}, \textrm{ }\alpha,\beta\in GL_2^+(\mathbb Q)
\end{equation} We are now ready to define
the quasimodular Hecke operators. 

\medskip

\begin{defn}\label{maindef} Let $\Gamma=\Gamma(N)$ be a principal congruence subgroup. A quasimodular Hecke 
operator of level $\Gamma$  is a function of finite support:
\begin{equation}
F:\Gamma\backslash GL_2^+(\mathbb Q)\longrightarrow \mathcal{QM} \qquad \Gamma\alpha
\mapsto F_\alpha
\end{equation} such that for any $\gamma\in \Gamma$, we have:
\begin{equation}\label{tt2.16}
F_{\alpha\gamma}=F_\alpha||\gamma 
\end{equation} The collection of all quasimodular Hecke operators of level $\Gamma$ will be denoted
by $\mathcal Q(\Gamma)$. 
\end{defn}

\medskip
We will now introduce the product structure on $\mathcal Q(\Gamma)$. In fact, we will introduce
two separate product structures $(\mathcal Q(\Gamma),\ast)$ and $(\mathcal Q(\Gamma),\ast^r)$ on
$\mathcal Q(\Gamma)$. 

\medskip
\begin{thm}(a) Let $\Gamma=\Gamma(N)$ be a principal congruence subgroup and let $\mathcal Q(\Gamma)$
be the collection of quasimodular Hecke operators of level $\Gamma$. Then, the product defined
by:
\begin{equation}\label{2.11}
(F\ast G)_\alpha:=\sum_{\beta\in \Gamma\backslash GL_2^+(\mathbb Q)}F_{\beta}\cdot (G_{\alpha\beta^{-1}}||\beta)\qquad \forall\textrm{ }\alpha\in GL_2^+(\mathbb Q)
\end{equation} for all $F$, $G\in \mathcal Q(\Gamma)$ makes $\mathcal Q(\Gamma)$ into an associative algebra. 

\medskip
(b) Let $\Gamma=\Gamma(N)$ be a principal congruence subgroup and let $\mathcal Q(\Gamma)$
be the collection of quasimodular Hecke operators of level $\Gamma$. Then, the product defined
by:
\begin{equation}\label{2.11vv}
(F\ast^r G)_\alpha:=\sum_{\beta\in \Gamma\backslash SL_2(\mathbb Z)}F_{\beta}\cdot (G_{\alpha\beta^{-1}}||\beta)\qquad \forall\textrm{ }\alpha\in GL_2^+(\mathbb Q)
\end{equation} for all $F$, $G\in \mathcal Q(\Gamma)$ makes $\mathcal Q(\Gamma)$ into an associative algebra which we denote by $\mathcal Q^r(\Gamma)$. 
\end{thm}

\begin{proof} (a) We need to check that the product in \eqref{2.11} is associative. First of all, we note
that the expression in \eqref{2.11} can be rewritten as:
\begin{equation}\label{2.12}
(F\ast G)_\alpha = \sum_{\alpha_2\alpha_1=\alpha} F_{\alpha_1}\cdot G_{\alpha_2}||\alpha_1
\qquad\forall\textrm{ }\alpha\in GL_2^+(\mathbb Q)
\end{equation} where the sum in \eqref{2.12} is taken over all pairs $(\alpha_1,\alpha_2)$ with $\alpha_2\alpha_1=\alpha$ modulo the following equivalence relation:
\begin{equation}
(\alpha_1,\alpha_2)\sim (\gamma\alpha_1,\alpha_2\gamma^{-1}) \qquad\forall\textrm{ }
\gamma\in \Gamma
\end{equation} Hence, for $F$, $G$, $H\in \mathcal Q(\Gamma)$, we can write:
\begin{equation}\label{2.14}
\begin{array}{ll}
(F\ast (G\ast H))_\alpha & =\sum_{\alpha'_2\alpha_1=\alpha}F_{\alpha_1}\cdot (G\ast H)_{\alpha'_2}||\alpha_1 \\
& =\sum_{\alpha'_2\alpha_1=\alpha}F_{\alpha_1}\cdot (\sum_{\alpha_3\alpha_2=\alpha'_2}
G_{\alpha_2}\cdot H_{\alpha_3}||\alpha_2)||\alpha_1 \\
& =\sum_{\alpha_3\alpha_2\alpha_1=\alpha}F_{\alpha_1}\cdot (G_{\alpha_2}||\alpha_1)\cdot 
(H_{\alpha_3}||\alpha_2\alpha_1)\\ 
\end{array}
\end{equation} where the sum in \eqref{2.14} is taken over all triples $(\alpha_1,\alpha_2,\alpha_3)$ with $\alpha_3\alpha_2\alpha_1=\alpha$ modulo the following equivalence relation:
\begin{equation}\label{2.15}
(\alpha_1,\alpha_2,\alpha_3)\sim (\gamma\alpha_1,\gamma'\alpha_2\gamma^{-1},\alpha_3\gamma'^{-1})\qquad \forall\textrm{ }\gamma,\gamma'\in \Gamma
\end{equation} On the other hand, we have
\begin{equation}\label{2.16}
\begin{array}{ll}
((F\ast G)\ast H)_\alpha & =\sum_{\alpha_3\alpha''_2=\alpha}(F\ast G)_{\alpha''_2}\cdot 
H_{\alpha_3}||\alpha''_2 \\
& =\sum_{\alpha_3\alpha''_2=\alpha} (\sum_{\alpha_2\alpha_1=\alpha''_2} F_{\alpha_1}
\cdot G_{\alpha_2}||\alpha_1)\cdot H_{\alpha_3}||\alpha''_2 \\
&=\sum_{\alpha_3\alpha_2\alpha_1=\alpha}F_{\alpha_1}\cdot (G_{\alpha_2}||\alpha_1)
\cdot (H_{\alpha_3}||\alpha_2\alpha_1) \\
\end{array}
\end{equation} where the sum in \eqref{2.16} is taken over all triples $(\alpha_1,\alpha_2,\alpha_3)$ with $\alpha_3\alpha_2\alpha_1=\alpha$ modulo the  equivalence relation in \eqref{2.15}. From 
\eqref{2.14} and \eqref{2.16} the result follows. 
 
\end{proof}

\medskip
We know that modular forms are quasimodular forms of depth $0$, i.e., for any $k\geq 0$, $N\geq 1$,
we have $\mathcal M_k(\Gamma(N))=\mathcal{QM}_k^0(\Gamma(N))$. It follows that the modular
tower $\mathcal M$ defined in \eqref{mtower} embeds into the quasimodular tower $\mathcal{QM}$
defined in \eqref{2.8}. We are now ready to show that the modular Hecke algebra 
$\mathcal A(\Gamma)$ of Connes and
Moscovici embeds into the quasimodular Hecke algebra $\mathcal Q(\Gamma)$ for 
any congruence subgroup $\Gamma=\Gamma(N)$. 

\medskip
\begin{thm} Let $\Gamma=\Gamma(N)$ be a principal congruence subgroup of $SL_2(\mathbb Z)$. 
Let $\mathcal A(\Gamma)$ be the modular Hecke algebra of level $\Gamma$ as defined in 
Definition \ref{CMdef} and let $\mathcal Q(\Gamma)$ be the quasimodular Hecke algebra
of level $\Gamma$ as defined in Definition \ref{maindef}. 
Then, there is a natural embedding of algebras $\mathcal A(\Gamma)\hookrightarrow 
\mathcal Q(\Gamma)$. 
\end{thm}

\begin{proof} For any $\alpha\in GL_2^+(\mathbb Q)$ and any $f\in \mathcal{QM}^s_k(\Gamma)$, we
consider the operation $f\mapsto f||\alpha$ as defined in \eqref{t2.11}:
\begin{equation}\label{2.24}
f||\alpha=\underset{i=0}{\overset{s}{\sum}}(a_i(f)|_{k-2i}\alpha)G_2^i \in \mathcal{QM}
\end{equation} In particular, if $f\in \mathcal M_k(\Gamma)=\mathcal{QM}^0_k(\Gamma)$ is a modular form, it follows
from \eqref{2.24} that:
\begin{equation}\label{2.25}
f||\alpha = a_0(f)|_k\alpha  = f|_k\alpha = f|\alpha \in \mathcal M
\end{equation} Hence, using the embedding of $\mathcal M$ in $\mathcal{QM}$, it follows from 
\eqref{tt2.11} in the definition of $\mathcal A(\Gamma)$ and from \eqref{tt2.16} in the definition
of $\mathcal Q(\Gamma)$ that we have an embedding $\mathcal A(\Gamma)\hookrightarrow 
\mathcal Q(\Gamma)$ of modules.  Further, we recall from \cite[$\S$ 1]{CM1} that the product on 
$\mathcal A(\Gamma)$ is given by:
\begin{equation}\label{kittyS}
(F\ast G)_\alpha:=\sum_{\beta\in \Gamma\backslash GL_2^+(\mathbb Q)}F_{\beta}\cdot (G_{\alpha\beta^{-1}}|\beta)\qquad \forall\textrm{ }\alpha\in GL_2^+(\mathbb Q), \textrm{ }F,G\in \mathcal A(\Gamma)
\end{equation} Comparing \eqref{kittyS} with the product on $\mathcal Q(\Gamma)$ described in \eqref{2.11}
and using \eqref{2.25} it follows that $\mathcal A(\Gamma)\hookrightarrow 
\mathcal Q(\Gamma)$ is an embedding of algebras. 

\end{proof}

\medskip
We end this section by describing the action of the algebra $\mathcal Q(\Gamma)$ on
 $\mathcal{QM}(\Gamma)$. 

\begin{thm} Let $\Gamma=\Gamma(N)$ be a principal congruence subgroup and let $\mathcal Q(\Gamma)$
be the algebra of quasimodular Hecke operators of level $\Gamma$.  Then, for any element
$f\in \mathcal{QM}(\Gamma)$ the action of $\mathcal Q(\Gamma)$ defined by:
\begin{equation}\label{2.28}
F\ast f:=\sum_{\beta\in \Gamma\backslash GL_2^+(\mathbb Q)}F_\beta \cdot f||\beta\qquad
\forall\textrm{ }F\in \mathcal Q(\Gamma)
\end{equation} makes $\mathcal{QM}(\Gamma)$ into a left module over $\mathcal Q(\Gamma)$. 
\end{thm}
\begin{proof} It is easy to check that the right hand side of \eqref{2.28} is independent of the choice
of coset representatives. Further, since $F\in \mathcal Q(\Gamma)$ is a function of finite support, we can choose finitely many coset representatives $\{\beta_1,\beta_2,...,\beta_n\}$ such that
\begin{equation}\label{tv2.29}
F\ast f=\sum_{j=1}^n F_{\beta_j}\cdot f||\beta_j
\end{equation} It suffices to consider the case $f\in \mathcal{QM}_k^s(\Gamma)$ for some
weight $k$ and depth $s$. Then, we can express $f$ as a sum: 
\begin{equation}
f=\sum_{i=0}^s a_i(f)G_2^i
\end{equation} where each $a_i(f)\in \mathcal M_{k-2i}(\Gamma)$. Similarly, for any $\beta\in GL_2^+(\mathbb Q)$, we can express $F_{\beta}$ as a finite sum:
\begin{equation}
F_{\beta}=\sum_{r=0}^{t_\beta}a_{\beta r}(F_{\beta})\cdot G_2^r 
\end{equation} with each $a_{\beta r}(F_{\beta})\in \mathcal M$.  In particular, we let $t=max\{t_{\beta_1},t_{\beta_2},...,t_{\beta_n}\}$ and we can now write:
\begin{equation}
F_{\beta_j}=\sum_{r=0}^ta_{\beta_j r}(F_{\beta_j})\cdot G_2^r
\end{equation}by adding appropriately many terms with zero coefficients in the expression for each
$F_{\beta_j}$.  Further, for any $\gamma\in \Gamma$, we know that $F_{\beta_j\gamma}
=F_{\beta_j}||\gamma=\sum_{r=0}^t(a_{\beta_jr}(F_{\beta_j})|\gamma)\cdot G_2^r$. In other words, we have, for each $j$:
\begin{equation}\label{2.33cv}
F_{\beta_j\gamma}=\sum_{r=0}^ta_{\beta_j\gamma r}(F_{\beta_j\gamma})\cdot G_2^r \qquad
a_{\beta_j\gamma r}(F_{\beta_j\gamma})=(a_{\beta_jr}(F_{\beta_j})|\gamma)
\end{equation} The sum in \eqref{tv2.29} can now be expressed as:
\begin{equation}
F\ast f:=\sum_{j=1}^nF_{\beta_j} \cdot f||\beta_j
=\sum_{i=0}^s\sum_{r=0}^{t}\sum_{j=1}^na_{\beta_jr}(F_{\beta_j})\cdot (a_i(f)|\beta_j)\cdot G_2^{r+i}\end{equation} For any $i$, $r$,  we now set:
\begin{equation}\label{2.35cv}
A_{ir}(F,f):=\sum_{j=1}^na_{\beta_jr}(F_{\beta_j})\cdot (a_i(f)|\beta_j)\end{equation} Again, it is easy to see that the sum $A_{ir}(F,f)$ in \eqref{2.35cv} does not depend on the choice of the coset representatives $\{\beta_1,\beta_2,...,\beta_n\}$. Then, for any $\gamma\in \Gamma$, we have:
\begin{equation}\label{2.36}
A_{ir}(F,f)|\gamma=\sum_{j=1}^n (a_{\beta_jr}(F_{\beta_j})|\gamma)\cdot (a_i(f)|\beta_j\gamma)
=\sum_{j=1}^n a_{\beta_j\gamma r}(F_{\beta_j\gamma}) \cdot (a_i(f)|\beta_j\gamma)=A_{ir}(F,f)
\end{equation} where the last equality in \eqref{2.36} follows from the fact that 
$\{\beta_1\gamma,\beta_2\gamma,...,\beta_n\gamma\}$ is another collection of  distinct cosets reprsentatives of 
$\Gamma$ in $GL_2^+(\mathbb Q)$. From \eqref{2.36}, we note that each $A_{ir}(F,f)\in \mathcal M(\Gamma)$. Then, the sum:
\begin{equation}
F\ast f=\sum_{i=0}^s\sum_{r=0}^tA_{ir}(F,f)\cdot G_2^{i+r}
\end{equation} is an element of $\mathcal{QM}(\Gamma)$. Hence, $\mathcal{QM}(\Gamma)$ is a left module over $\mathcal Q(\Gamma)$. 
\end{proof}

\medskip

\medskip

\section{The Lie algebra and Hopf algebra actions on $\bf \mathcal Q(\Gamma)$}

\medskip

\medskip
Let $\Gamma=\Gamma(N)$ be a principal congruence subgroup of $SL_2(\mathbb Z)$.  
In this section, we will describe two different sets of operators on the collection $\mathcal Q(\Gamma)$ of quasimodular Hecke operators
of level $\Gamma$. Given a quasimodular form $f\in \mathcal{QM}(\Gamma)$ of level $\Gamma$, we have mentioned
in the last section that $f$ can be expressed as a finite sum:
\begin{equation}\label{3.1ner}
f=\sum_{i=0}^sa_i(f)\cdot G_2^i
\end{equation} where $G_2$ is the classical Eisenstein series of weight $2$ and each $a_i(f)$ is a modular
form of level $\Gamma$. Then in Section 3.1, we consider operators on the quasimodular tower that
act on the powers of $G_2$ appearing in \eqref{3.1ner}. These induce operators $D$, $\{T_k^l\}_{k\geq 1,l\geq 0}$
on the collection $\mathcal Q(\Gamma)$ of quasimodular Hecke operators
of level $\Gamma$. In order to understand the action of these operators on products of elements
in $\mathcal Q(\Gamma)$, we also need to define extra operators $\{\phi^{(m)}\}_{m\geq 1}$. Finally, we show
that these operators may all be described in terms of a Hopf algebra $\mathcal H$ with a ``Hopf action''
on $\mathcal Q^r(\Gamma)$, i.e.,
\begin{equation}
h(F^1\ast^r F^2)=\sum h_{(1)}(F^1)\ast^r h_{(2)}(F^2) \qquad\forall \textrm{ } h\in \mathcal H, \textrm{ }F^1,F^2\in \mathcal Q^r(\Gamma)
\end{equation} where the coproduct $\Delta:\mathcal H\longrightarrow \mathcal H\otimes \mathcal H$ is given by
$\Delta(h)=\sum h_{(1)}\otimes h_{(2)}$ for any $h\in \mathcal H$. In Section 3.2, we consider operators on the quasimodular tower 
$\mathcal{QM}$ that act on the modular coefficients $a_i(f)$ appearing in \eqref{3.1ner}. These induce on $\mathcal Q(\Gamma)$
analogues of operators acting on the modular Hecke algebra $\mathcal A(\Gamma)$ of Connes and Moscovici \cite{CM1}. Then, we
show that $\mathcal Q(\Gamma)$ carries a Hopf action of the same Hopf algebra $\mathcal H_1$ of codimension
$1$ foliations that acts on $\mathcal A(\Gamma)$.

\medskip

\medskip
\subsection{The operators $\bf D$, $\bf \{T^l_k\}$ and $\bf \{\phi^{(m)}\}$ on $\mathcal Q(\Gamma)$}

\medskip

\medskip
 For any even number $K\geq 2$, let $G_K$ be the classical Eisenstein series of weight $K$ as in 
 \eqref{2.12ov}. Since $G_2$ is a quasimodular form, i.e., $G_2\in \mathcal{QM}$, its derivative $G_2'\in \mathcal{QM}$. Further, it is well known that:
\begin{equation}\label{3.1}
G_2'=\frac{5\pi i}{3}G_4 - 4\pi iG_2^2
\end{equation} where $G_4$ is the Eisenstein series of weight $4$ (which is a modular form). For our
purposes, it will be convenient to write:
\begin{equation}
G_2'=\sum_{j=0}^2g_jG_2^j
\end{equation} with each $g_j$ a modular form. From \eqref{3.1}, it follows that:
\begin{equation}\label{3.3}
g_0=\frac{5\pi i}{3}G_4 \qquad g_1=0 \qquad g_2=-4\pi i
\end{equation} We are now ready to define the operators $D$ and $\{W_k\}_{k\geq 1}$ on $\mathcal{QM}$. The first operator $D$ differentiates
the powers of $G_2$:   \begin{equation}\label{3.4}
\begin{array}{c}
D:\mathcal{QM}\longrightarrow \mathcal{QM} \\
\begin{array}{ll}
 f=\underset{i=0}{\overset{s}{\sum}}a_i(f)G_2^i \mapsto & -\frac{1}{8\pi i}\left( \underset{i=0}{\overset{s}{\sum}}ia_i(f)G_2^{i-1}\cdot G_2' \right)\\
 & = -\frac{1}{8\pi i}\underset{i=0}{\overset{s}{\sum}}\underset{j=0}{\overset{2}{\sum}} ia_i(f)g_jG_2^{i+j-1}
 \end{array}
\end{array}
\end{equation} The operators  $\{W_k\}_{k\geq 1}$ are ``weight operators'' and $W_k$ also steps up the power of $G_2$ by $k-2$. We set:
\begin{equation}\label{3.5x}
W_k:\mathcal{QM}\longrightarrow \mathcal{QM} \qquad 
 f=\underset{i=0}{\overset{s}{\sum}}a_i(f)G_2^i \mapsto  \underset{i=0}{\overset{s}{\sum}}ia_i(f)G_2^{i+k-2}
\end{equation} From the definitions in \eqref{3.4} and \eqref{3.5x}, we can easily check that $D$ and
$W_k$ are derivations on $\mathcal{QM}$. 
Finally, for any $\alpha\in GL_2^+(\mathbb Q)$ and any integer $m\geq 1$, we set 
\begin{equation}\label{nu}
\nu_\alpha^{(m)}=-\frac{5}{24}\left( G_4^m|\alpha - G_4^m\right)
\end{equation} 

\medskip
\begin{lem}\label{L3.1} (a) Let $f\in \mathcal{QM}$ be an element of the quasimodular tower and 
$\alpha\in GL_2^+(\mathbb Q)$. Then, the operator $D$ satisfies:
\begin{equation}\label{3.6}
D(f)||\alpha=D(f||\alpha) + \nu_\alpha^{(1)}\cdot (W_1(f)||\alpha)
\end{equation} where, using \eqref{nu}, we know that $\nu_\alpha^{(1)}$ is given by:
\begin{equation}\label{3.7}
\nu_\alpha^{(1)}:= -\frac{1}{8\pi i}(g_0|\alpha - g_0)=-\frac{5}{24}\left( G_4|\alpha - G_4\right)\qquad \forall\textrm{ }
\alpha\in GL_2^+(\mathbb Q)
\end{equation}

\medskip
(b) For $f\in\mathcal{QM}$ and $\alpha\in GL_2^+(\mathbb Q)$, each operator $W_k$, $k\geq 1$ satisfies:
\begin{equation}\label{3.8}
W_k(f)||\alpha=W_k(f||\alpha)
\end{equation}

\end{lem}

\begin{proof} We start by proving part (a). For the sake of definiteness, we assume that $ f=\underset{i=0}{\overset{s}{\sum}}a_i(f)G_2^i$
with each $a_i(f)\in \mathcal M$. For $\alpha\in GL_2^+(\mathbb Q)$, it follows from \eqref{3.4} that:
\begin{equation}\label{3.8z}
\begin{array}{lllll}
D(f)||\alpha & =  -\frac{1}{8\pi i}\left(\underset{i}{\sum}\underset{j}{\sum} ia_i(f)g_jG_2^{i+j-1}\right)||\alpha & \textrm{ }& D(f||\alpha) & =  D\left(\underset{i}{\sum} (a_i(f)|\alpha)G_2^i \right)\\ 
& = -\frac{1}{8\pi i} \underset{i}{\sum}\underset{j}{\sum}\textrm{ } i(a_i(f)|\alpha)(g_j|\alpha)G_2^{i+j-1} & & & = -\frac{1}{8\pi i}\underset{i}{\sum}
\underset{j}{\sum}\textrm{ }i(a_i(f)|\alpha)g_jG_2^{i+j-1} \end{array}
\end{equation} From \eqref{3.8z} it follows that:
\begin{equation}\label{3.9}
D(f)||\alpha - D(f||\alpha)= -\frac{1}{8\pi i}\sum_{i=0}^s\sum_{j=0}^2 \textrm{ }i(a_i(f)|\alpha)(g_j|\alpha - g_j)G_2^{i+j-1}
\end{equation} From \eqref{3.3}, it is clear that $g_j|\alpha - g_j=0$ for $j=1$ and $j=2$. It follows that:
\begin{equation*}
D(f)||\alpha - D(f||\alpha)= -\frac{1}{8\pi i}\sum_{i=0}^s\textrm{ }i(a_i(f)|\alpha)(g_0|\alpha - g_0)G_2^{i-1}= -\frac{1}{8\pi i}(g_0|\alpha 
- g_0)\cdot \left (\sum_{i=0}^s\textrm{ }i(a_i(f)|\alpha) G_2^{i-1}\right)
\end{equation*} This proves the result of (a). The result of part (b) is clear from the definition in \eqref{3.5x}.

\end{proof}

\medskip
We note here that it follows from \eqref{nu} that for any $\alpha$, $\beta\in GL_2^+(\mathbb Q)$, we have:
\begin{equation}\label{3.12}
\nu_{\alpha\beta}^{(m)}=\nu_{\alpha}^{(m)}|\beta + \nu_\beta^{(m)} \qquad \forall \textrm{ }m\geq 1
\end{equation} Additionally, since each $G_4^m$ is a modular form, we know that when $\alpha\in SL_2(\mathbb Z)$:
\begin{equation}\label{3.125}
\nu_\alpha^{(m)}=-\frac{5}{24}(G_4^m|\alpha - G_4^m) = 0 \qquad \forall\textrm{ }
\alpha\in SL_2(\mathbb Z), m\geq 1
\end{equation} Moreover, from the definitions in \eqref{3.4} and \eqref{3.5x} respectively, it is easily verified that $D$ and $\{W_k\}_{k\geq 1}$
are derivations on the quasimodular tower $\mathcal{QM}$. We now proceed to define operators
on the quasimodular Hecke algebra $\mathcal Q(\Gamma)$ for some principal congruence subgroup
$\Gamma=\Gamma(N)$. Choose $F\in \mathcal Q(\Gamma)$. We set:
\begin{equation}\label{3.13}
\begin{array}{c}
D,W_k,\phi^{(m)}: \mathcal Q(\Gamma)\longrightarrow \mathcal Q(\Gamma) \quad k\geq 1, m\geq 1\\
D(F)_\alpha: = D(F_\alpha) \quad W_k(F)_\alpha :=W_k(F_\alpha) \quad \phi^{(m)}(F)_\alpha:=\nu_\alpha^{(m)}
\cdot F_\alpha \qquad \forall \textrm{ }\alpha\in GL_2^+(\mathbb Q) \\
\end{array}
\end{equation} From Lemma \ref{L3.1} and the properties of $\nu_\alpha^{(m)}$ described in 
\eqref{3.12} and \eqref{3.125}, it may be easily verified that the operators 
$D$, $W_k$ and $\phi^{(m)}$ in \eqref{3.13} are well defined on $\mathcal Q(\Gamma)$. 
We will now compute the commutators of the operators $D$, $\{W_k\}_{k\geq 1}$ and $\{\phi^{(m)}\}_{m\geq 1}$ on
$\mathcal Q(\Gamma)$. In order to describe these commutators, we need one more operator $E$:
\begin{equation}\label{3.19T}
E:\mathcal{QM}\longrightarrow \mathcal{QM}\qquad f\mapsto G_4\cdot f
\end{equation} Since $G_4$ is a modular form of level $\Gamma(1)=SL_2(\mathbb Z)$, i.e., 
$G_4|\gamma=G_4$ for any $\gamma\in SL_2(\mathbb Z)$, it is clear that $E$ induces a
well defined operator on $\mathcal Q(\Gamma)$:
\begin{equation}
E:\mathcal Q(\Gamma)\longrightarrow \mathcal Q(\Gamma) \qquad E(F)_\alpha:=E(F_\alpha)=G_4\cdot F_\alpha\quad \forall\textrm{ }F\in 
\mathcal Q(\Gamma), \alpha\in GL_2^+(\mathbb Q)
\end{equation} We will now describe the commutator relations between the operators  $D$, $E$, $\{E^lW_k\}_{k\geq 1,l\geq 0}$ and $\{\phi^{(m)}\}_{m\geq 1}$ on $\mathcal Q(\Gamma)$. 

\medskip
\begin{thm}\label{P3.3} Let $\Gamma =\Gamma(N)$ be a principal congruence subgroup and let $\mathcal Q(\Gamma)$ be the algebra of quasimodular Hecke operators of level $\Gamma$. The operators $D$, $E$, $\{E^lW_k\}_{k\geq 1,l\geq 0}$ and $\{\phi^{(m)}\}_{m\geq 1}$ on $\mathcal Q(\Gamma)$ satisfy the following relations:
\begin{equation}\label{3.19}
\begin{array}{c}
[E,E^lW_k]=0 \quad [E,D]=0 \quad [E,\phi^{(m)}]=0 \quad [D,\phi^{(m)}]=0 \quad [W_k,\phi^{(m)}]=0 \quad [\phi^{(m)},
\phi^{(m')}]=0 \\
\mbox{$[E^lW_k,D]=\frac{5}{24}(k-1)(E^{l+1}W_{k-1})- \frac{1}{2}(k-3)E^lW_{k+1}$}\\
\end{array}
\end{equation} 
\end{thm}

\begin{proof} For any $F\in \mathcal Q(\Gamma)$ and any $\alpha\in GL_2^+(\mathbb Q)$, by definition, we know that $D(F)_\alpha=D(F_\alpha)$, $W_k(F)_\alpha=W_k(F_\alpha)$ and 
$E(F)_\alpha=E(F_\alpha)$.  Hence, in order to prove that $[E,W_k]=0$ and
$[E,D]=0$, it suffices to show that $[E,W_k](f)=0$ and $[E,D](f)=0$ respectively  for any element $f\in 
\mathcal{QM}$. Both of these are easily verified from the definitions of $D$ and $W_k$ in 
\eqref{3.4} and \eqref{3.5x} respectively. Further, since $[E,W_k]=0$, it is clear that
$[E,E^lW_k]=0$.

\medskip Similarly, in order to prove the expression for $[E^lW_k,D]$, it suffices to prove that:
\begin{equation}
[E^lW_k,D](f) = \frac{5}{24}(k-1)(E^{l+1}W_{k-1})(f)- \frac{1}{2}(k-3)E^lW_{k+1}(f)
\end{equation} for any $f\in \mathcal{QM}$. Further, it suffices to consider the case where $f=\sum_{i=0}^sa_i(f)G_2^i$ where the
$a_i(f)\in \mathcal M$. We now have:
\begin{equation}\label{3.20}
\begin{array}{c}
W_kD(f) =  -\frac{1}{8\pi i}W_k\left(\underset{i=0}{\overset{s}{\sum}}\underset{j=0}{\overset{2}{\sum}} ia_i(f)g_jG_2^{i+j-1} \right) = -\frac{1}{8\pi i}\underset{i=0}{\overset{s}{\sum}}\underset{j=0}{\overset{2}{\sum}} i(i+j-1)a_i(f)g_jG_2^{i+j+k-3}\\
DW_k(f)=D\left( \underset{i=0}{\overset{s}{\sum}}ia_i(f)G_2^{i+k-2}\right)=  -\frac{1}{8\pi i}\underset{i=0}{\overset{s}{\sum}}\underset{j=0}{\overset{2}{\sum}} i(i+k-2)a_i(f)g_jG_2^{i+j+k-3}\\ 
\end{array}
\end{equation} It follows from \eqref{3.20} that:
\begin{equation*}\label{3.21}
\begin{array}{ll}
[W_k,D](f)&=-\frac{1}{8\pi i}\underset{i=0}{\overset{s}{\sum}}\underset{j=0}{\overset{2}{\sum}} ija_i(f)g_jG_2^{i+j+k-3} +\frac{1}{8\pi i}\underset{i=0}{\overset{s}{\sum}}\underset{j=0}{\overset{2}{\sum}} i(k-1)a_i(f)g_jG_2^{i+j+k-3} \\ & =-\frac{2g_2}{8\pi i}\underset{i=0}{\overset{s}{\sum}}ia_i(f)G_2^{i+k-1}+(k-1)\frac{1}{8\pi i}\underset{i=0}{\overset{s}{\sum}}ia_i(f)g_0G_2^{i+k-3}
+(k-1)\frac{g_2}{8\pi i}\underset{i=0}{\overset{s}{\sum}}ia_i(f)G_2^{i+k-1} \\
\end{array}
\end{equation*} where the second equality uses the fact that $g_1=0$. Further, since $g_0=\frac{5\pi i}{3}G_4$ and $g_2=-4\pi i$, it follows from \eqref{3.21} that we have:
\begin{equation}\label{3.24cv}
\begin{array}{l}
[W_k,D](f)=\frac{5}{24}(k-1)\underset{i=0}{\overset{s}{\sum}}iG_4a_i(f)G_2^{i+k-3}
-\frac{1}{2}(k-3)\underset{i=0}{\overset{s}{\sum}}ia_i(f)G_2^{i+k-1} \\
=\frac{5}{24}(k-1)(EW_{k-1})(f) - \frac{1}{2}(k-3)W_{k+1}(f)\\
\end{array}
\end{equation} Finally, since $E$ commutes with $\{W_k\}_{k\geq 1}$ and $D$, it follows
from \eqref{3.24cv} that:
\begin{equation}
[E^lW_k,D]=\frac{5}{24}(k-1)(E^{l+1}W_{k-1})- \frac{1}{2}(k-3)E^lW_{k+1} \qquad \forall\textrm{ }k\geq 1,l\geq 0
\end{equation} as operators on $\mathcal Q(\Gamma)$. Finally, it may be easily verified from the definitions that
$[E,\phi^{(m)}]= [D,\phi^{(m)}]= [W_k,\phi^{(m)}]=0$. 

\end{proof} 

\medskip The operators $\{E^lW_k\}_{k\geq 1,l\geq 0}$ appearing in Proposition \ref{P3.3} above can be described
more succintly as:
\begin{equation}
T^l_k:\mathcal{QM}\longrightarrow \mathcal{QM} \qquad T^l_k:=E^lW_k \qquad \forall\textrm{ }k\geq 1, l\geq 0
\end{equation} and
\begin{equation}
T^l_k:\mathcal{Q}(\Gamma)\longrightarrow \mathcal{Q}(\Gamma) \qquad T^l_k(F)_\alpha:=T^l_k(F_\alpha)=E^lW_k(F_\alpha)
 \qquad \forall\textrm{ }F\in \mathcal Q(\Gamma),\alpha\in GL_2^+(\mathbb Q)
\end{equation}
We are now ready to describe the Lie algebra action on $\mathcal Q(\Gamma)$.

\medskip
\begin{thm}\label{excp} Let $\mathcal L$ be the   Lie algebra generated by the symbols $D$, $\{T^l_k\}_{k\geq 1, l\geq 0}$, 
$\{\phi^{(m)}\}_{m\geq 1}$ along with the following relations between the commutators: \begin{equation}\label{3.27}
\begin{array}{c}
[T_k^l,T_{k'}^{l'}]=(k'-k)T_{k+k'-2}^{l+l'}\\
\mbox{$[D,\phi^{(m)}]=0$} \qquad [T^l_k,\phi^{(m)}]=0 \qquad [\phi^{(m)},
\phi^{(m')}]=0 \\
\mbox{$[T^l_k,D]$}=  \frac{5}{24}(k-1)T_{k-1}^{l+1}- \frac{1}{2}(k-3)T^l_{k+1}\\
\end{array}
\end{equation} Then, for any principal congruence subgroup $\Gamma=\Gamma(N)$, we have a Lie action of $\mathcal L$ on
the algebra of quasimodular Hecke operators $\mathcal Q(\Gamma)$ of level $\Gamma$.
\end{thm}

\begin{proof} For any $k\geq 1$ and $l\geq 0$, $T^l_k$ has been defined to be the operator $E^lW_k$ on $\mathcal Q(\Gamma)$. 
We want to verify that:
\begin{equation}\label{3.27EX}
[T_k^l,T_{k'}^{l'}]=(k-k')T_{k+k'-2}^{l+l'}\qquad\forall\textrm{ }k,k'\geq 1,\textrm{ }l,l'\geq 0
\end{equation} As in the proof of Proposition \ref{P3.3}, it suffices to show that the relation in \eqref{3.27EX}
holds for any $f\in \mathcal{QM}$. As before, we let $f=\sum_{i=0}^sa_i(f)G_2^i$ where each $a_i(f)\in \mathcal M$. We now have:
\begin{equation}\label{3.28EX}
\begin{array}{c}
T_k^lT_{k'}^{l'}(f)=T_k^l\left(\underset{i=0}{\overset{s}{\sum}}ia_i(f)G_4^{l'}\cdot G_2^{i+k'-2}\right)=\underset{i=0}{\overset{s}{\sum}}i(i+k'-2)a_i(f)G_4^{l+l'}\cdot G_2^{i+k'+k-4}\\
T_{k'}^{l'}T_k^l(f)=T_{k'}^{l'}\left(\underset{i=0}{\overset{s}{\sum}}ia_i(f)G_4^{l}\cdot G_2^{i+k-2}\right)=\underset{i=0}{\overset{s}{\sum}}i(i+k-2)a_i(f)G_4^{l+l'}\cdot G_2^{i+k'+k-4}\\
\end{array}
\end{equation} From \eqref{3.28EX} it follows that:
\begin{equation}
[T_k^l,T_{k'}^{l'}](f)=(k'-k)\underset{i=0}{\overset{s}{\sum}}ia_i(f)G_4^{l+l'}\cdot G_2^{i+k'+k-4}=(k'-k)T_{k+k'-2}^{l+l'}
\end{equation} Hence, the relation \eqref{3.27EX} holds for the operators $T_k^l$, $T_{k'}^{l'}$ acting on
$\mathcal Q(\Gamma)$.
The remaining relations in \eqref{3.27}  for  the Lie action of $\mathcal L$  on $\mathcal Q(\Gamma)$ follow from  \eqref{3.19}.
\end{proof}

\medskip

\begin{lem}\label{L3.5} Let $f\in \mathcal{QM}$ be an element of the quasimodular tower and let $\alpha\in GL_2^+(\mathbb Q)$. Then, for
any $k\geq1 $, $l\geq 0$, the operator $T^l_k:\mathcal{QM}
\longrightarrow \mathcal{QM}$ satisfies:
\begin{equation}
T^l_k(f)||\alpha = T^l_k(f||\alpha)-\frac{24}{5}\nu_{\alpha}^{(l)}\cdot (T^0_k(f)||\alpha)
\end{equation}
\end{lem}

\begin{proof} For the sake of definiteness, we assume that $f=\sum_{i=0}^sa_i(f)\cdot G_2^i$ with each $a_i(f)\in \mathcal M$. We now compute:

\begin{equation}
\begin{array}{lll}
T^l_k(f)||\alpha = (E^lW_k)(f)||\alpha& \qquad \qquad & T^l_k(f||\alpha) = (E^lW_k)(f||\alpha) \\
= \left(\underset{i=0}{\overset{s}{\sum}}iG_4^l\cdot a_i(f)G_2^{i+k-2}\right)||\alpha & & 
=(E^lW_k)\left(\underset{i=0}{\overset{s}{\sum}}(a_i(f)|\alpha)G_2^i\right)\\
= \underset{i=0}{\overset{s}{\sum}}i(G_4^l|\alpha)\cdot (a_i(f)|\alpha)G_2^{i+k-2} & & 
=\underset{i=0}{\overset{s}{\sum}}i(G_4^l)\cdot (a_i(f)|\alpha)G_2^{i+k-2}\\
\end{array}
\end{equation} 
Subtracting, it follows that:
\begin{equation}
T^l_k(f)||\alpha - T^l_k(f||\alpha) = (G_4^l|\alpha - G_4^l)\cdot \left( \underset{i=0}{\overset{s}{\sum}}i(a_i(f)|\alpha)
G_2^{i+k-2}\right)=-\frac{24}{5}\nu_{\alpha}^{(l)}\cdot (W_k(f)||\alpha)
\end{equation} Putting $T^0_k=E^0W_k=W_k$, we have the result.

\end{proof} 

\medskip

\begin{thm}\label{Prp3.6} Let $\Gamma=\Gamma(N)$ be a principal congruence subgroup and let $\mathcal Q(\Gamma)$ be the algebra
of quasimodular Hecke operators of level $\Gamma$. Then, for any $k\geq 1$, $l\geq 0$, the operator
$T_k^l$ satisfies:
\begin{equation}\label{3.32}
T_k^l(F^1\ast F^2)=T_k^l(F^1)\ast F^2 + F^1\ast T_k^l(F^2)+\frac{24}{5}  (\phi^{(l)}(F^1)\ast T_k^0(F^2))_\alpha \qquad \forall \textrm{ } F^1,F^2 \in \mathcal Q(\Gamma)
\end{equation} Further, the operators $\{T_k^l\}_{k\geq 1,l\geq 0}$ are all derivations on the algebra
$\mathcal Q^r(\Gamma)=(\mathcal Q(\Gamma),\ast^r)$. 
\end{thm}

\begin{proof} We know that $T_k^l= E^lW_k$ and that $W_k$ is a derivation on $\mathcal{QM}$. We choose quasimodular
Hecke operators $F^1$, $F^2\in \mathcal Q(\Gamma)$. Then, for any $\alpha\in GL_2^+(\mathbb Q)$, we know that:
\begin{equation*}
\begin{array}{l}
T_k^l(F^1\ast F^2)_\alpha =E^lW_k\left(\underset{\beta\in \Gamma \backslash GL_2^+(\mathbb Q)}{\sum} F^1_\beta
\cdot (F^2_{\alpha\beta^{-1}}||\beta)\right) \\
= \underset{\beta\in \Gamma \backslash GL_2^+(\mathbb Q)}{\sum}  E^lW_k\left(F^1_\beta
\cdot (F^2_{\alpha\beta^{-1}}||\beta) \right) \\
 = \underset{\beta\in \Gamma \backslash GL_2^+(\mathbb Q)}{\sum} G_4^l\cdot W_k(F^1_\beta)\cdot (F^2_{\alpha\beta^{-1}}||\beta)+\underset{\beta\in \Gamma \backslash GL_2^+(\mathbb Q)}{\sum} F^1_\beta\cdot G_4^l\cdot  W_k(F^2_{\alpha\beta^{-1}}||\beta)\\
 = \underset{\beta\in \Gamma \backslash GL_2^+(\mathbb Q)}{\sum} G_4^l\cdot W_k(F^1_\beta)\cdot (F^2_{\alpha\beta^{-1}}||\beta)+\underset{\beta\in \Gamma \backslash GL_2^+(\mathbb Q)}{\sum} F^1_\beta\cdot G_4^l\cdot  (W_k(F^2_{\alpha\beta^{-1}})||\beta)\\
=(T_k^l(F^1)\ast F^2)_\alpha + \underset{\beta\in \Gamma \backslash GL_2^+(\mathbb Q)}{\sum} F^1_\beta\cdot (G_4^l|\beta)\cdot (W_k(F^2_{\alpha\beta^{-1}})||\beta)- \underset{\beta\in \Gamma \backslash GL_2^+(\mathbb Q)}{\sum} F^1_\beta\cdot  (G_4^l|\beta-G_4^l) \cdot (W_k(F^2_{\alpha\beta^{-1}})||\beta)\\
 = (T_k^l(F^1)\ast F^2)_\alpha +  (F^1\ast T_k^l(F^2))_\alpha +\frac{24}{5}  \underset{\beta\in \Gamma \backslash GL_2^+(\mathbb Q)}{\sum} F^1_\beta\cdot  \nu_\beta^{(l)} \cdot (W_k(F^2_{\alpha\beta^{-1}})||\beta) \\
 = (T_k^l(F^1)\ast F^2)_\alpha +  (F^1\ast T_k^l(F^2))_\alpha +\frac{24}{5}  (\phi^{(l)}(F^1)\ast T_k^0(F^2))_\alpha\\
\end{array}
\end{equation*} where it is understood that $\phi^{(0)}=0$.  This proves \eqref{3.32}. Further, since $\nu_\beta^{(l)}=0$ for any $\beta\in SL_2(\mathbb Z)$, 
when we consider the product $\ast^r$ defined in \eqref{2.11vv} on the algebra $\mathcal Q^r(\Gamma)$, the calculation above reduces to
\begin{equation}
T_k^l(F^1\ast^r F^2)=T_k^l(F^1)\ast^r F^2 + F^1\ast^r T_k^l(F^2)
\end{equation} Hence, each $T_k^l$ is a derivation on $\mathcal Q^r(\Gamma)$. 
\end{proof}

\medskip
\begin{thm}\label{Prp3.2} Let $\Gamma =\Gamma(N)$ be a principal congruence subgroup and let $\mathcal Q(\Gamma)$ be the algebra of quasimodular Hecke operators of level $\Gamma$. 

\medskip
(a) The operator $D:\mathcal Q(\Gamma)\longrightarrow \mathcal Q(\Gamma)$ on the algebra
$(\mathcal Q(\Gamma),\ast)$ satisfies:
\begin{equation}\label{3.15}
D(F^1\ast F^2)=D(F^1)\ast F^2 + F^1\ast D(F^2) -\phi^{(1)}(F^1)\ast T_1^0(F^2) \qquad\forall\textrm{ }F^1,F^2\in \mathcal Q(\Gamma)
\end{equation} When we consider the product $\ast^r$, the operator $D$ becomes a derivation
on the algebra $\mathcal Q^r(\Gamma)=(\mathcal Q(\Gamma),\ast^r)$, i.e.:
\begin{equation}\label{3.151z}
D(F^1\ast^r F^2)=D(F^1)\ast^r F^2 + F^1\ast^r D(F^2)  \qquad\forall\textrm{ }F^1,F^2\in \mathcal Q^r(\Gamma)
\end{equation}

\medskip
(b) The operators $\{W_k\}_{k\geq 1}$ and $\{\phi^{(m)}\}_{m\geq 1}$ are derivations on $\mathcal Q(\Gamma)$, i.e., 
\begin{equation}\label{3.152z}
\begin{array}{c}
W_k(F^1\ast F^2)=W_k(F^1)\ast F^2 + F^1\ast W_k(F^2)  \\
\phi^{(m)}(F^1\ast F^2)=\phi^{(m)}(F^1)\ast F^2 + F^1\ast \phi^{(m)}(F^2) \\
\end{array}
\end{equation} for any $F^1$, $F^2\in \mathcal Q(\Gamma)$. Additionally, $\{\phi^{(m)}\}_{m\geq 1}$ and
$\{W_k\}_{k\geq 1}$ are also derivations on the algebra $\mathcal Q^r(\Gamma)=(\mathcal Q(\Gamma),\ast^r)$. 

\end{thm}

\begin{proof} (a) We choose quasimodular Hecke operators $F^1$, $F^2\in \mathcal Q(\Gamma)$. We have mentioned before that $D$ is a derivation on $\mathcal{QM}$.  Then,
for any $\alpha\in GL_2^+(\mathbb Q)$, we have:
\begin{equation*}
\begin{array}{ll}
D(F^1\ast F^2)_\alpha& =D\left(\underset{\beta\in \Gamma \backslash GL_2^+(\mathbb Q)}{\sum} F^1_\beta
\cdot (F^2_{\alpha\beta^{-1}}||\beta)\right) \\
& = \underset{\beta\in \Gamma \backslash GL_2^+(\mathbb Q)}{\sum}  D\left(F^1_\beta
\cdot (F^2_{\alpha\beta^{-1}}||\beta) \right) \\
& = \underset{\beta\in \Gamma \backslash GL_2^+(\mathbb Q)}{\sum} D(F^1_\beta)\cdot (F^2_{\alpha\beta^{-1}}||\beta)+\underset{\beta\in \Gamma \backslash GL_2^+(\mathbb Q)}{\sum} F^1_\beta\cdot D(F^2_{\alpha\beta^{-1}}||\beta)\\
& =(D(F^1)\ast F^2)_\alpha + \underset{\beta\in \Gamma \backslash GL_2^+(\mathbb Q)}{\sum} F^1_\beta\cdot (D(F^2_{\alpha\beta^{-1}})||\beta)- \underset{\beta\in \Gamma \backslash GL_2^+(\mathbb Q)}{\sum} F^1_\beta\cdot \nu_\beta^{(1)} \cdot (W_1(F^2_{\alpha\beta^{-1}})||\beta)\\
& = (D(F^1)\ast F^2)_\alpha +  (F^1\ast D(F^2))_\alpha -  (\phi^{(1)}(F^1)\ast T_1^0(F^2))_\alpha\\
\end{array}
\end{equation*} This proves \eqref{3.15}. In order to prove \eqref{3.151z}, we note that 
$\nu_\beta^{(1)}=0$ for any $\beta\in SL_2(\mathbb Z)$ (see \eqref{3.125}). Hence, when we use
the product $\ast^r$ defined in \eqref{2.11vv},  the calculation above
reduces to
\begin{equation}
D(F^1\ast^r F^2)=D(F^1)\ast^r F^2 + F^1\ast^r D(F^2)
\end{equation} for any $F^1$, $F^2\in \mathcal Q^r(\Gamma)$. 

\medskip
(b) For any $F^1$, $F^2\in \mathcal Q(\Gamma)$ and knowing from \eqref{3.12} that
$\nu_\alpha^{(m)}=\nu_\beta^{(m)}+\nu_{\alpha\beta^{-1}}^{(m)}|\beta$, we have:
\begin{equation}
\begin{array}{ll}
\phi^{(m)}(F^1\ast F^2)_\alpha& =\nu_\alpha^{(m)}\cdot \underset{\beta\in \Gamma \backslash GL_2^+(\mathbb Q)}{\sum}   F^1_\beta
\cdot (F^2_{\alpha\beta^{-1}}||\beta)  \\
& = \underset{\beta\in \Gamma \backslash GL_2^+(\mathbb Q)}{\sum}  (\nu_\beta^{(m)} \cdot F^1_\beta)
\cdot (F^2_{\alpha\beta^{-1}}||\beta) +  \underset{\beta\in \Gamma \backslash GL_2^+(\mathbb Q)}{\sum}  F^1_\beta
\cdot (\nu_{\alpha\beta^{-1}}^{(m)}|\beta)\cdot (F^2_{\alpha\beta^{-1}}||\beta) \\
& = \phi^{(m)}(F^1)\ast F^2 + \underset{\beta\in \Gamma \backslash GL_2^+(\mathbb Q)}{\sum}  F^1_\beta
\cdot  ((\nu_{\alpha\beta^{-1}}^{(m)}\cdot F^2_{\alpha\beta^{-1}})||\beta)\\
& = \phi^{(m)}(F^1)\ast F^2 + \underset{\beta\in \Gamma \backslash GL_2^+(\mathbb Q)}{\sum}  F^1_\beta
\cdot  (\phi^{(m)}(F^2)_{\alpha\beta^{-1}}||\beta)\\ 
& = \phi^{(m)}(F^1)\ast F^2 + F^1\ast \phi^{(m)}(F^2) \\
\end{array}
\end{equation} The fact that each $W_k$ is also a derivation on $\mathcal Q(\Gamma)$ now follows from a similar calculation using the fact that $W_k$ is a derivation on the quasimodular tower $\mathcal{QM}$ and that $W_k(f)||\alpha=W_k(f||\alpha)$ for any $f\in \mathcal{QM}$, $\alpha\in GL_2^+(\mathbb Q)$ (from 
\eqref{3.8}). Finally, a similar calculation may be used to verify that $\{W_k\}_{k\geq 1}$ and $\{\phi^{(m)}\}_{m\geq 1}$ are all derivations
on $\mathcal Q^r(\Gamma)$. 
\end{proof}

\medskip
We now introduce the Hopf algebra $\mathcal H$ that acts on $\mathcal Q^r(\Gamma)$. The Hopf  algebra $\mathcal H$ is the
universal enveloping algebra $\mathcal U(\mathcal L)$ of the Lie algebra $\mathcal L$ defined  by generators
$D$, $\{T_k^l\}_{k\geq 1,l\geq 0}$, $\{\phi^{(m)}\}_{m\geq 1}$ satisfying the following relations:
\begin{equation}
\begin{array}{c}
[T_k^l,T_{k'}^{l'}]=(k'-k)T_{k+k'-2}^{l+l'}
\qquad 
\mbox{$[T^l_k,D]$}=  \frac{5}{24}(k-1)T_{k-1}^{l+1}- \frac{1}{2}(k-3)T^l_{k+1}\\
\mbox{$[D,\phi^{(m)}]=0$} \qquad [T^l_k,\phi^{(m)}]=0 \qquad [\phi^{(m)},
\phi^{(m')}]=0 \\
\end{array}
\end{equation}
As such, the  coproduct $\Delta:
\mathcal H\longrightarrow \mathcal H\otimes \mathcal H$ is defined by:
\begin{equation}\label{3.34}
\Delta(D)=D\otimes 1+1\otimes D \qquad 
\Delta(T_k^l)=T_k^l\otimes 1+1\otimes T_k^l \qquad
\Delta(\phi^{(m)})=\phi^{(m)}\otimes 1 +1\otimes \phi^{(m)}
\end{equation} 
We will now show that $\mathcal H$ has a Hopf action on the algebra $\mathcal Q^r(\Gamma)$.

\medskip
\begin{thm} Let $\Gamma=\Gamma(N)$ be a principal congruence subgroup of $SL_2(\mathbb Z)$. Then, there is a Hopf action of $\mathcal H$ on
the algebra $\mathcal Q^r(\Gamma)$, i.e., 
\begin{equation}\label{3.44}
h(F^1\ast^r F^2)=\sum h_{(1)}(F^1)\ast^r h_{(2)}(F^2) \qquad\forall\textrm{ }F^1, F^2\in \mathcal Q^r(\Gamma),
\textrm{ }h\in \mathcal H
\end{equation} where $\Delta(h)=\sum h_{(1)}\otimes h_{(2)}$ for any $h\in \mathcal H$. 
\end{thm}

\begin{proof} In order to prove \eqref{3.44}, it suffices to verify the relation for $D$ and each of $\{T_k^l\}_{k\geq 1,l\geq 0}$,
$\{\phi^{(m)}\}_{m\geq 1}$. From Proposition \ref{Prp3.6} and Proposition \ref{Prp3.2},
we know that for $F^1$, $F^2\in \mathcal Q^r(\Gamma)$ and any $k\geq 1$, $l\geq 0$, $m\geq 1$:
\begin{equation}
\begin{array}{c}
D(F^1\ast^r F^2)=D(F^1)\ast^r F^2 +F^1\ast^r D(F^2)\\
T_k^l(F^1\ast^r F^2)=T_k^l(F^1)\ast^r F^2+F^1\ast^r T_k^l(F^2) \\ 
\phi^{(m)}(F^1\ast^r F^2)=\phi^{(m)}(F^1)\ast^r F^2+F^1\ast^r \phi^{(m)}(F^2) \\
\end{array}
\end{equation} Comparing with the expressions for the coproduct in \eqref{3.34}, it is clear that \eqref{3.44}
holds for each $h\in \mathcal H$. 
\end{proof}

\medskip

\medskip
\subsection{The operators $\bf X$, $\bf Y$ and $\bf \{\delta_n\}$ of Connes and Moscovici}

\medskip

\medskip
Let $\Gamma=\Gamma(N)$ be a congruence subgroup. In this subsection, we will show that the 
algebra $\mathcal Q(\Gamma)$ carries an action of the Hopf algebra $\mathcal H_1$ of Connes and Moscovici \cite{CM0}. The Hopf algebra $\mathcal H_1$ is part of a larger family 
$\{\mathcal H_n\}_{n\geq 1}$ of Hopf algebras
defined in \cite{CM0} and $\mathcal H_1$ is the Hopf algebra corresponding to ``codimension 
1 foliations''. As an algebra, $\mathcal H_1$ is identical to the universal enveloping algebra
$\mathcal U(\mathcal L_1)$ of the Lie algebra $\mathcal L_1$ generated by $X$, $Y$, $\{\delta_n\}_{n\geq 1}$ satisfying the commutator relations:
\begin{equation}\label{3.2.46}
[Y,X]=X \quad [X,\delta_n]=\delta_{n+1} \quad [Y,\delta_n]=n\delta_n\quad [\delta_k,\delta_l]=0\quad\forall \textrm{ }k,l,n\geq 1
\end{equation} Further, the coproduct $\Delta:\mathcal H_1\longrightarrow 
\mathcal H_1\otimes \mathcal H_1$ on $\mathcal H_1$ is determined by:
\begin{equation}\label{3.2.47qz}
\begin{array}{c}
\Delta(X)=X\otimes 1+1\otimes X+\delta_1\otimes Y \\
\Delta(Y)=Y\otimes 1+1\otimes Y\qquad \Delta(\delta_1)=\delta_1\otimes 1+1\otimes\delta_1 \\
\end{array}
\end{equation} Finally, the antipode $S:\mathcal H_1\longrightarrow \mathcal H_1$ is given by:
\begin{equation}\label{3.2.48qz}
S(X)=-X+\delta_1Y\qquad S(Y)=-Y \qquad S(\delta_1)=-\delta_1
\end{equation} Following Connes and Moscovici \cite{CM1}, we define the operators $X$ and
$Y$ on the modular tower: for any congruence subgroup $\Gamma=\Gamma(N)$, we set:
\begin{equation}
Y:\mathcal M_k(\Gamma)\longrightarrow \mathcal M_k(\Gamma)\qquad Y(f):=\frac{k}{2}f \qquad \forall
\textrm{ }f\in \mathcal M_k(\Gamma)
\end{equation} Further, the operator $X:
\mathcal M_k(\Gamma)\longrightarrow \mathcal M_{k+2}(\Gamma)$ is the Ramanujan differential operator on modular forms:
\begin{equation}
X(f):=\frac{1}{2\pi i}\frac{d}{dz}(f)-\frac{1}{12\pi i}\frac{d}{dz}(\log \Delta)\cdot Y(f) \qquad \forall
\textrm{ }f\in \mathcal M_k(\Gamma)
\end{equation} where $\Delta(z)$ is the well known modular form of weight $12$ given by:
\begin{equation}
\Delta(z)=(2\pi )^{12}q\prod_{n=1}^\infty (1-q^n)^{24}, \textrm{ }q=e^{2\pi iz}
\end{equation} We start by extending these operators to the quasimodular tower
$\mathcal{QM}$. Let $f\in \mathcal{QM}^s_k(\Gamma)$ be a quasimodular form. Then, we can
express $f=\underset{i=0}{\overset{s}{\sum}}a_i(f)G_2^i$ where $a_i(f)\in \mathcal M_{k-2i}(\Gamma)$. We set:
\begin{equation}\label{3.52}
X(f)=\sum_{i=0}^sX(a_i(f))\cdot G_2^i\qquad Y(f)=\sum_{i=0}^sY(a_i(f))\cdot G_2^i
\end{equation} From \eqref{3.52}, it is clear that $X$ and $Y$ are derivations on $\mathcal{QM}$. 

\medskip
\begin{lem}\label{L3.2.1} Let $f\in \mathcal{QM}$ be an element of the quasimodular tower. Then, for any $\alpha\in GL_2^+(\mathbb Q)$, we have:
\begin{equation}\label{3.2.53}
X(f)||\alpha = X(f||\alpha)+(\mu_{\alpha^{-1}}\cdot Y(f))||\alpha
\end{equation} where, for any $\delta\in GL_2^+(\mathbb Q)$, we set:
\begin{equation}\label{3.2.54}
\mu_\delta:=\frac{1}{12\pi i}\frac{d}{dz}\log \frac{\Delta |\delta}{\Delta} 
\end{equation} Further, we have $Y(f||\alpha)=Y(f)||\alpha$. 
\end{lem}

\begin{proof} Following \cite[Lemma 5]{CM1}, we know that for any  $g\in \mathcal M$, we have:
\begin{equation}
X(g)|\alpha = X(g|\alpha)+(\mu_{\alpha^{-1}}\cdot Y(g))|\alpha\qquad\forall\textrm{ }\alpha\in GL_2^+(
\mathbb Q)
\end{equation} It suffices to consider the case $f\in \mathcal{QM}^s_k(\Gamma)$ for some congruence
subgroup $\Gamma$. If we express $f\in \mathcal{QM}^s_k(\Gamma)$ as $f=\underset{i=0}{\overset{s}{\sum}}a_i(f)G_2^i$ with $a_i(f)\in 
\mathcal M_{k-2i}(\Gamma)$, it follows that:
\begin{equation}\label{3.2.56}
X(a_i(f))|\alpha = X(a_i(f)|\alpha)+(\mu_{\alpha^{-1}}\cdot Y(a_i(f)))|\alpha\qquad\forall\textrm{ }\alpha\in GL_2^+(
\mathbb Q)
\end{equation} for each $0\leq i\leq s$. Combining \eqref{3.2.56} with the definitions of $X$
and $Y$ on the quasimodular tower in \eqref{3.52}, we can easily prove \eqref{3.2.53}. Finally, it
is clear from the definition of $Y$ that $Y(f||\alpha)=Y(f)||\alpha$. 

\end{proof}

\medskip
From the definition of $\mu_{\delta}$ in \eqref{3.2.54}, one may verify that (see \cite[$\S$ 3)]{CM1}):
\begin{equation}\label{3.2.57X}
\mu_{\delta_1\delta_2}=\mu_{\delta_1}|\delta_2 +\mu_{\delta_2} 
\qquad \forall\textrm{ }\delta_1,\delta_2\in GL_2^+(\mathbb Q)
\end{equation} and that $\mu_\delta=0$ for any $\delta\in SL_2(\mathbb Z)$. We now define  operators $X$, $Y$ and $\{\delta_n\}_{n\geq 1}$ on the quasimodular Hecke algebra $\mathcal Q(\Gamma)$
for some congruence subgroup $\Gamma=\Gamma(N)$. Let $F\in \mathcal Q(\Gamma)$ be a
quasimodular Hecke operator of level $\Gamma$; then we define operators:
\begin{equation}\label{3.2.58}
\begin{array}{c}
X, Y , \delta_n:\mathcal Q(\Gamma)\longrightarrow \mathcal Q(\Gamma) \\
X(F)_\alpha:=X(F_\alpha)\qquad Y(F)_\alpha:=Y(F_\alpha)\qquad 
\delta_n(F)_\alpha=X^{n-1}(\mu_{\alpha})\cdot F_\alpha \qquad \forall\textrm{ }\alpha\in 
GL_2^+(\mathbb Q) \\
\end{array}
\end{equation} We will now show that the Lie algebra
$\mathcal L_1$ with generators $X$, $Y$, $\{\delta_n\}_{n\geq 1}$ satisfying the commutator
relations in \eqref{3.2.46} acts on the algebra $\mathcal Q(\Gamma)$. Additionally, in order to give a Lie action on
the algebra $\mathcal Q^r(\Gamma)=(\mathcal Q(\Gamma),\ast^r)$, we define at this juncture the smaller Lie algebra
$\mathfrak l_1\subseteq \mathcal L_1$ with generators $X$ and $Y$ satisfying the relation
\begin{equation}\label{3.59uu}
[Y,X]=X
\end{equation} Further, we consider the Hopf algebra $\mathfrak h_1$ that arises as the universal
enveloping algebra $\mathcal U(\mathfrak l_1)$ of the Lie algebra $\mathfrak l_1$. We will
show that $\mathcal H_1$ (resp. $\mathfrak h_1$) has a Hopf action on the algebra
$\mathcal Q(\Gamma)$ (resp. $\mathcal Q^r(\Gamma))$. We start by describing the Lie actions. 

\medskip
\begin{thm}\label{P3.10} Let $\mathcal L_1$ be the Lie algebra with generators $X$, $Y$ and 
$\{\delta_n\}_{n\geq 1}$ satisfying the following commutator relations:
\begin{equation}\label{3.2.60}
[Y,X]=X \quad [X,\delta_n]=\delta_{n+1} \quad [Y,\delta_n]=n\delta_n\quad [\delta_k,\delta_l]=0\quad\forall \textrm{ }k,l,n\geq 1
\end{equation}
Then, for any given congruence subgroup $\Gamma=\Gamma(N)$  of
$SL_2(\mathbb Z)$, we have a Lie action of $\mathcal L_1$ on the module $\mathcal Q(\Gamma)$. 
\end{thm}

\begin{proof} From \cite[$\S$ 3]{CM1}, we know that for any element $g\in \mathcal M$ of the modular tower, we have $[Y,X](g)=X(g)$. Since the action of $X$ and $Y$ on the quasimodular tower
$\mathcal{QM}$ (see \eqref{3.52}) is naturally extended from their action on $\mathcal M$,
it follows that $[Y,X]=X$ on the quasimodular tower $\mathcal{QM}$. In particular, given any 
quasimodular Hecke operator $F\in \mathcal Q(\Gamma)$ and any $\alpha\in GL_2^+(\mathbb Q)$, we have $[Y,X](F_\alpha)=X(F_\alpha)$ for the element $F_\alpha\in\mathcal{QM}$. By definition,
$X(F)_\alpha=X(F_\alpha)$ and $Y(F_\alpha)=Y(F)_\alpha$ and hence $[Y,X]=X$ holds
for the action of $X$ and $Y$ on $\mathcal Q(\Gamma)$. 

\medskip
Further, since $X$ is a derivation on $\mathcal{QM}$ and $\delta_n(F)_\alpha=X^{n-1}(\mu_\alpha)\cdot F_\alpha$,
we have 
\begin{equation}
\begin{array}{ll}
[X,\delta_n](F)_\alpha & = X(X^{n-1}(\mu_{\alpha})\cdot F_\alpha) - X^{n-1}(\mu_{\alpha})\cdot X(F_\alpha)\\
& = X(X^{n-1}(\mu_{\alpha}))\cdot F_\alpha =X^n(\mu_{\alpha})\cdot F_\alpha=\delta_{n+1}(F)_\alpha \\
\end{array}
\end{equation} Similarly, since $\mu_\alpha
\in \mathcal M \subseteq \mathcal{QM}$ is of weight $2$ and $Y$ is a derivation on $\mathcal{QM}$, we have:
\begin{equation}
\begin{array}{ll}
[Y,\delta_n](F)_\alpha &=Y(X^{n-1}(\mu_{\alpha})\cdot F_\alpha) - X^{n-1}(\mu_{\alpha})\cdot Y(F_\alpha) \\
&=Y(X^{n-1}(\mu_{\alpha}))\cdot F_\alpha=nX^{n-1}(\mu_{\alpha})\cdot F_\alpha=n\delta_n(F)_\alpha \\
\end{array}
\end{equation} Finally, we can verify easily that $[\delta_k,\delta_l]=0$ for any $k$, $l\geq 1$. 
\end{proof}

\medskip

From Proposition \ref{P3.10}, it is also clear that the smaller Lie algebra $\mathfrak l_1\subseteq 
\mathcal L_1$ has a Lie action on the module $\mathcal Q(\Gamma)$. 

\medskip
\begin{lem}\label{L3.11} Let $\Gamma=\Gamma(N)$ be a congruence subgroup of $SL_2(\mathbb Z)$ and let $\mathcal Q(\Gamma)$ be the algebra of quasimodular Hecke operators of level $\Gamma$. Then, the operator
$X:\mathcal Q(\Gamma)\longrightarrow \mathcal Q(\Gamma)$ on the algebra
$(\mathcal Q(\Gamma),\ast)$ satisfies: \begin{equation} \label{3.2.63}
X(F^1\ast F^2)=X(F^1)\ast F^2 + F^1\ast X(F^2) +\delta_{1}(F^1)\ast Y(F^2) \qquad\forall\textrm{ }F^1,F^2\in \mathcal Q(\Gamma)
\end{equation} When we consider the product $\ast^r$, the operator $X$ becomes a derivation
on the algebra $\mathcal Q^r(\Gamma)=(\mathcal Q(\Gamma),\ast^r)$, i.e.:
\begin{equation} \label{3.2.64}
X(F^1\ast^r F^2)=X(F^1)\ast^r F^2 + F^1\ast^r X(F^2)  \qquad\forall\textrm{ }F^1,F^2\in \mathcal Q^r(\Gamma)
\end{equation}
\end{lem}

\begin{proof} We choose quasimodular Hecke operators $F^1$, $F^2\in \mathcal Q(\Gamma)$. Using \eqref{3.2.57X}, we also note that
\begin{equation}
0=\mu_1=\mu_{\beta^{-1}}|\beta +\mu_\beta \qquad \forall \textrm{ }\beta\in GL_2^+(\mathbb Q)
\end{equation}
We have mentioned before that $X$ is a derivation on $\mathcal{QM}$.  Then,
for any $\alpha\in GL_2^+(\mathbb Q)$, we have:
\begin{equation*}
\begin{array}{ll}
X(F^1\ast F^2)_\alpha& =X\left(\underset{\beta\in \Gamma \backslash GL_2^+(\mathbb Q)}{\sum} F^1_\beta
\cdot (F^2_{\alpha\beta^{-1}}||\beta)\right) \\
& = \underset{\beta\in \Gamma \backslash GL_2^+(\mathbb Q)}{\sum}  X\left(F^1_\beta
\cdot (F^2_{\alpha\beta^{-1}}||\beta) \right) \\
& = \underset{\beta\in \Gamma \backslash GL_2^+(\mathbb Q)}{\sum} X(F^1_\beta)\cdot (F^2_{\alpha\beta^{-1}}||\beta)+\underset{\beta\in \Gamma \backslash GL_2^+(\mathbb Q)}{\sum} F^1_\beta\cdot X(F^2_{\alpha\beta^{-1}}||\beta)\\
& =(X(F^1)\ast F^2)_\alpha + \underset{\beta\in \Gamma \backslash GL_2^+(\mathbb Q)}{\sum} F^1_\beta\cdot (X(F^2_{\alpha\beta^{-1}})||\beta)- \underset{\beta\in \Gamma \backslash GL_2^+(\mathbb Q)}{\sum} F^1_\beta\cdot ((\mu_{\beta^{-1}} \cdot Y(F^2_{\alpha\beta^{-1}}))||\beta)\\
& =(X(F^1)\ast F^2)_\alpha + (F^1\ast X(F^2))_\alpha + \underset{\beta\in \Gamma \backslash GL_2^+(\mathbb Q)}{\sum} (F^1_\beta\cdot \mu_{\beta})\cdot (Y(F^2_{\alpha\beta^{-1}})||\beta)\\
& = (X(F^1)\ast F^2)_\alpha +  (F^1\ast X(F^2))_\alpha + (\delta_1(F^1)\ast Y(F^2))_\alpha\\
\end{array}
\end{equation*} This proves \eqref{3.2.63}. In order to prove \eqref{3.2.64}, we note that 
$\mu_\beta=0$ for any $\beta\in SL_2(\mathbb Z)$. Hence, if we use the product $\ast^r$, the calculation above
reduces to
\begin{equation}\label{3.2.65}
X(F^1\ast^r F^2)=X(F^1)\ast^r F^2 + F^1\ast^r X(F^2)
\end{equation} for any $F^1$, $F^2\in \mathcal Q^r(\Gamma)$. 
\end{proof} 

\medskip
Finally, we describe the Hopf action of $\mathcal H_1$ on the algebra $(\mathcal Q(\Gamma),
\ast)$ as well as the Hopf action of $\mathfrak h_1$ on the algebra
$\mathcal Q^r(\Gamma)=(\mathcal Q(\Gamma),\ast^r)$. 

\medskip
\begin{thm} Let $\Gamma=\Gamma(N)$ be a  congruence subgroup of $SL_2(\mathbb Z)$. Then, the Hopf algebra $\mathcal H_1$ has a Hopf action on the quasimodular Hecke algebra 
$(\mathcal Q(\Gamma),\ast)$; in other words, we have:
\begin{equation}\label{3.2.66} 
h(F^1\ast F^2)=\sum h_{(1)}(F^1) \otimes h_{(2)}(F^2)\qquad \forall\textrm{ }h\in \mathcal H_1,
F^1,F^2\in \mathcal Q(\Gamma)
\end{equation} where the coproduct $\Delta:\mathcal H_1\longrightarrow \mathcal H_1
\otimes \mathcal H_1$ is given by $\Delta(h)=\sum h_{(1)}\otimes h_{(2)}$ for any
$h\in \mathcal H_1$. Similarly, there exists a Hopf action of the Hopf algebra
$\mathfrak h_1$ on the algebra $\mathcal Q^r(\Gamma)=(\mathcal Q(\Gamma),\ast^r)$. 
\end{thm}

\begin{proof} In order to prove \eqref{3.2.66}, it suffices to check the relation for
$X$, $Y$ and $\delta_1\in \mathcal H_1$. For the element $X\in \mathcal H_1$, this is already the result
of Lemma \ref{L3.11}. Now, for any $F^1$, $F^2\in \mathcal Q(\Gamma)$ and $\alpha\in GL_2^+(\mathbb Q)$, we have:
\begin{equation}\label{3.2.67}
\begin{array}{ll}
\delta_1(F^1\ast F^2)_\alpha& = \mu_\alpha\cdot\left(\underset{\beta\in \Gamma\backslash GL_2^+(\mathbb Q)}{\sum}F^1_\beta
\cdot (F^2_{\alpha\beta^{-1}}||\beta)\right)\\
&=\underset{\beta\in \Gamma\backslash GL_2^+(\mathbb Q)}{\sum}(\mu_\beta\cdot F^1_\beta)\cdot (F^2_{\alpha\beta^{-1}}||\beta) + \underset{\beta\in \Gamma\backslash GL_2^+(\mathbb Q)}{\sum}
F^1_\beta\cdot ((\mu_{\alpha\beta^{-1}}\cdot F^2_{\alpha\beta^{-1}})||\beta)\\
& = (\delta_1(F^1)\ast F^2)_\alpha + (F^1\ast \delta_1(F^2))_\alpha \\
\end{array}
\end{equation} Further, using the fact that $Y$ is a derivation on $\mathcal{QM}$ and
$Y(f||\alpha)=Y(f)||\alpha$ for any $f\in \mathcal{QM}$, $\alpha\in GL_2^+(\mathbb Q)$, we can
easily verify the relation \eqref{3.2.66} for the element $Y\in \mathcal H_1$. This proves
\eqref{3.2.66} for all $h\in \mathcal H_1$. 

\medskip
Finally, in order to demonstrate the Hopf action of $\mathfrak h_1$ on $\mathcal Q^r(\Gamma)$,
we need to check that:
\begin{equation}
X(F^1\ast^r F^2)=X(F^1)\ast^r F^2 + F^1\ast^r X(F^2)\qquad Y(F^1\ast^r F^2)=Y(F^1)\ast^r F^2 + F^1\ast^r Y(F^2)
\end{equation} for any $F^1$, $F^2\in \mathcal Q^r(\Gamma)$. The relation for $X$ has already been proved in \eqref{3.2.65}. The relation for $Y$
is again an easy consequence of the fact that $Y$ is a derivation on $\mathcal{QM}$ and
$Y(f||\alpha)=Y(f)||\alpha$ for any $f\in \mathcal{QM}$, $\alpha\in GL_2^+(\mathbb Q)$. 
\end{proof}

\medskip

\medskip
\section{Twisted Quasimodular Hecke operators}

\medskip

\medskip
Let $\Gamma=\Gamma(N)$ be a principal congruence subgroup of $SL_2(\mathbb Z)$. For any
$\sigma\in SL_2(\mathbb Z)$, we have developed the  theory of $\sigma$-twisted modular Hecke operators in \cite{AB1}. 
In this section, we introduce and study the collection $\mathcal Q_\sigma(\Gamma)$ of quasimodular Hecke operators
of level $\Gamma$ twisted by $\sigma$. When $\sigma=1$, $\mathcal Q_\sigma(\Gamma)$ coincides with the algebra 
$\mathcal Q(\Gamma)$ of quasimodular Hecke operators. In general, we will show that $\mathcal Q_\sigma(\Gamma)$ is a right
$\mathcal Q(\Gamma)$-module and carries a pairing:
\begin{equation}\label{4.1pb}
(\_\_,\_\_):\mathcal Q_\sigma(\Gamma)\otimes \mathcal Q_\sigma(\Gamma)\longrightarrow 
\mathcal Q_\sigma(\Gamma)
\end{equation} We recall from Section 3 the Lie algebra $\mathfrak{l}_1$ with two generators $Y$, $X$ satisfying 
$[Y,X]=X$. If we let $\mathfrak{h}_1$ be the Hopf algebra that is  the universal enveloping algebra 
of $\mathfrak{l}_1$, we show in Section 4.1 that the pairing in \eqref{4.1pb} on $\mathcal Q_\sigma(\Gamma)$ carries a
``Hopf action'' of $\mathfrak{h}_1$. In other words, we have:
\begin{equation}
h(F^1,F^2)=\sum (h_{(1)}(F^1),h_{(2)}(F^2))\qquad\forall \textrm{ } h\in \mathfrak{h}_1, \textrm{ }F^1,F^2\in \mathcal Q_\sigma(\Gamma)
\end{equation} where the coproduct $\Delta:\mathfrak{h}_1\longrightarrow \mathfrak{h}_1\otimes \mathfrak{h}_1$ is given by
$\Delta(h)=\sum h_{(1)}\otimes h_{(2)}$ for any $h\in \mathfrak{h}_1$.  In Section 4.2, we consider operators
$X_\tau:\mathcal Q_\sigma(\Gamma)\longrightarrow \mathcal Q_{\tau\sigma}(\Gamma)$ for any $\tau$, $\sigma\in SL_2(\mathbb Z)$. In particular, we consider operators
acting between the levels of the graded module:
\begin{equation}
\mathbb Q_\sigma(\Gamma)=\bigoplus_{m\in \mathbb Z}\mathcal Q_{\sigma(m)}(\Gamma)
\end{equation} where for any $\sigma \in SL_2(\mathbb Z)$, we set $\sigma(m)
=\begin{pmatrix} 1 & m \\ 0 & 1 \\ \end{pmatrix}\cdot \sigma$. Further, we generalize the  pairing 
on $\mathcal Q_\sigma(\Gamma)$ in \eqref{4.1pb} to a pairing:
\begin{equation}\label{4.4pb}
(\_\_,\_\_):\mathcal Q_{\sigma(m)}(\Gamma)\otimes \mathcal Q_{\sigma(n)}(\Gamma)\longrightarrow \mathcal Q_{\sigma(m+n)}(\Gamma)
\qquad \forall\textrm{ }m,n\in \mathbb Z
\end{equation} We show that the pairing in \eqref{4.4pb} is a special case of a more general pairing 
\begin{equation}
(\_\_,\_\_):\mathcal Q_{\tau_1\sigma}(\Gamma)\otimes \mathcal Q_{\tau_2\sigma}(\Gamma)\longrightarrow \mathcal Q_{\tau_1\tau_2\sigma}(\Gamma)
\end{equation} where $\tau_1$, $\tau_2$  are commuting matrices in $SL_2(\mathbb Z)$. From \eqref{4.4pb}, it is clear
that we have a graded pairing on 
$\mathbb Q_{\sigma}(\Gamma)$ that extends the pairing on $\mathcal Q_{\sigma}(\Gamma)$. 
Finally, we consider 
the Lie algebra $\mathfrak{l}_{\mathbb Z}$ with generators $\{Z,X_n|n\in \mathbb Z\}$ satisfying the 
commutator relations:
\begin{equation}
[Z,X_n]=(n+1)X_n\qquad [X_n,X_{n'}]=0\qquad \forall\textrm{ }n,n'\in \mathbb Z
\end{equation} Then, for $n=0$, we have $[Z,X_0]=X_0$ and hence the Lie algebra $\mathfrak{l}_{\mathbb Z}$ contains
the Lie algebra $\mathfrak{l}_1$ acting on $\mathcal Q_\sigma(\Gamma)$. Then, if we let $\mathfrak{h}_{\mathbb Z}$ be the Hopf algebra that is the universal enveloping
algebra of $\mathfrak{l}_{\mathbb Z}$, we show that $\mathfrak{h}_{\mathbb Z}$ has a Hopf action on the 
pairing on $\mathbb Q_\sigma(\Gamma)$. In other words, for any $F^1$, $F^2\in \mathbb Q_{\sigma}(\Gamma)$, we have
\begin{equation}
h(F^1,F^2)=\sum (h_{(1)}(F^1),h_{(2)}(F^2))\qquad \forall\textrm{ }h\in \mathfrak{h}_{\mathbb Z}
\end{equation} where the coproduct $\Delta:\mathfrak h_{\mathbb Z}
\longrightarrow \mathfrak h_{\mathbb Z}\otimes \mathfrak h_{\mathbb Z}$ is defined
by setting $\Delta(h):=\sum h_{(1)}\otimes h_{(2)}$ for each $h\in \mathfrak h_{\mathbb Z}$. 

\medskip

\medskip

\subsection{The pairing on $\mathcal Q_\sigma(\Gamma)$ and Hopf action}

\medskip

\medskip
Let $\sigma\in SL_2(\mathbb Z)$ and let $\Gamma=\Gamma(N)$ be a principal congruence subgroup
of $SL_2(\mathbb Z)$. We start by defining the collection 
$\mathcal Q_\sigma(\Gamma)$ of  quasimodular Hecke operators of level $\Gamma$ 
twisted by $\sigma$. When $\sigma=1$, this reduces to the definition of
$\mathcal Q(\Gamma)$. 

\medskip
\begin{defn} Choose $\sigma\in SL_2(\mathbb Z)$ and let $\Gamma=\Gamma(N)$ be a principal congruence
subgroup of $SL_2(\mathbb Z)$. A $\sigma$-twisted quasimodular Hecke operator
$F$ of level $\Gamma$ is a function of finite support:
\begin{equation}\label{4.1}
F:\Gamma\backslash GL_2^+(\mathbb Q)\longrightarrow \mathcal{QM}\qquad
\Gamma\alpha\mapsto F_\alpha\in \mathcal{QM}
\end{equation} such that:
\begin{equation}\label{4.2}
F_{\alpha\gamma}=F_\alpha||\sigma\gamma\sigma^{-1}\qquad\forall\textrm{ }\gamma\in\Gamma
\end{equation} We denote by $\mathcal Q_\sigma(\Gamma)$ the collection of $\sigma$-twisted
quasimodular Hecke operators of level $\Gamma$. 
\end{defn}

\medskip
\begin{thm} Let $\Gamma=\Gamma(N)$ be a principal congruence subgroup of $SL_2(\mathbb Z)$
and choose some $\sigma\in SL_2(\mathbb Z)$. Then there exists a pairing:
\begin{equation}\label{4.3}
(\_\_,\_\_):\mathcal Q_\sigma(\Gamma)\otimes \mathcal Q_\sigma(\Gamma)
\longrightarrow \mathcal Q_\sigma(\Gamma)
\end{equation} defined as follows:
\begin{equation}\label{4.4}
(F^1,F^2)_{\alpha}:=\underset{\beta\in \Gamma\backslash SL_2(\mathbb Z)}{\sum}
F_{\beta\sigma}^1\cdot (F^2_{\alpha\sigma^{-1}\beta^{-1}}||\sigma\beta)
\qquad \forall\textrm{ }F^1, F^2\in \mathcal Q_\sigma(\Gamma),\alpha\in GL_2^+(\mathbb Q)
\end{equation}
\end{thm}

\begin{proof} We choose $\gamma\in \Gamma$. Then, for any $\beta\in SL_2(\mathbb Z)$, we have:
\begin{equation}
F^1_{\gamma\beta\sigma}=F^1_{\beta\sigma}\qquad 
F^2_{\alpha\sigma^{-1}\beta^{-1}\gamma^{-1}}||\sigma\gamma\beta = F^2_{\alpha\sigma^{-1}\beta^{-1}}||\sigma\gamma^{-1}\sigma^{-1}\sigma\gamma\beta= 
F^2_{\alpha\sigma^{-1}\beta^{-1}}||\sigma\beta
\end{equation} and hence the sum in \eqref{4.4} is well defined, i.e., it does not depend
on the choice of coset representatives. We have to show that $(F^1,F^2)\in \mathcal Q_\sigma(\Gamma)$. For this, we first note that $F^2_{\gamma\alpha\sigma^{-1}\beta^{-1}}= F^2_{\alpha\sigma^{-1}\beta^{-1}}$ for any $\gamma\in \Gamma$ and hence from the expression
in \eqref{4.4}, it follows that $(F^1,F^2)_{\gamma\alpha}=(F^1,F^2)_\alpha$. On the other hand, for any $\gamma\in \Gamma$,
we can write:
\begin{equation}\label{4.6}
(F^1,F^2)_{\alpha\gamma}=\underset{\beta\in \Gamma\backslash SL_2(\mathbb Z)}{\sum}
F^1_{\beta\sigma}\cdot (F^2_{\alpha\gamma\sigma^{-1}\beta^{-1}}||\sigma\beta)
\end{equation} We put $\delta=\beta\sigma\gamma^{-1}\sigma^{-1}$. It is clear that as $\beta$ runs through all the coset representatives of $\Gamma$ in $SL_2(\mathbb Z)$, so does $\delta$. From \eqref{4.2}, we know that $F^1_{\delta\sigma\gamma}=F^1_{\delta\sigma}||\sigma\gamma\sigma^{-1}$. Then, we can rewrite \eqref{4.6} as:
\begin{equation}
\begin{array}{ll}
(F^1,F^2)_{\alpha\gamma} & = \underset{\delta\in \Gamma\backslash SL_2(\mathbb Z)}{\sum}
F^1_{\delta\sigma\gamma}\cdot (F^2_{\alpha\sigma^{-1}\delta^{-1}}||\sigma\delta\sigma\gamma\sigma^{-1})\\
& = \underset{\delta\in \Gamma\backslash SL_2(\mathbb Z)}{\sum}
(F^1_{\delta\sigma}||\sigma\gamma\sigma^{-1})\cdot ((F^2_{\alpha\sigma^{-1}\delta^{-1}}||\sigma\delta)||\sigma\gamma\sigma^{-1})\\
& = \left(\underset{\delta\in \Gamma\backslash SL_2(\mathbb Z)}{\sum}
F_{\delta\sigma}^1\cdot (F^2_{\alpha\sigma^{-1}\delta^{-1}}||\sigma\delta)\right)||(\sigma\gamma\sigma^{-1})\\
&=(F^1,F^2)_\alpha||\sigma\gamma\sigma^{-1}\\
\end{array}
\end{equation} It follows that $(F^1,F^2)\in \mathcal Q_\sigma(\Gamma)$ and hence we have a well defined pairing $(\_\_,\_\_):\mathcal Q_\sigma(\Gamma)\otimes \mathcal Q_\sigma(\Gamma)
\longrightarrow \mathcal Q_\sigma(\Gamma)$. 

\end{proof}

\medskip
We now consider the Hopf algebra $\mathfrak h_1$ defined in Section 3.2. By definition, $\mathfrak h_1$ is the universal enveloping algebra of the Lie algebra $\mathfrak l_1$ with two generators
$X$ and $Y$ satisfying $[Y,X]=X$. We will now show that $\mathfrak l_1$ has a Lie action 
on $\mathcal Q_\sigma(\Gamma)$ and that $\mathfrak h_1$ has a ``Hopf action'' with respect to the pairing on $\mathcal Q_\sigma(\Gamma)$. 

\medskip
\begin{thm}\label{Prop4.3} Let $\sigma\in SL_2(\mathbb Z)$ and let $\Gamma=\Gamma(N)$ be a principal
congruence subgroup of $SL_2(\mathbb Z)$. 

\medskip
(a) The Lie algebra $\mathfrak l_1$ has a Lie action on $\mathcal Q_\sigma(\Gamma)$ defined
by: 
\begin{equation}\label{4.8}
X(F)_\alpha:=X(F_\alpha) \qquad Y(F)_\alpha:=Y(F_\alpha)\qquad\forall\textrm{ }F\in \mathcal Q_\sigma(\Gamma), \alpha\in GL_2^+(\mathbb Q)
\end{equation} 

\medskip
(b) The universal enveloping algebra $\mathfrak h_1$ of the Lie algebra $\mathfrak l_1$ has a ``Hopf action'' with respect to the pairing on $\mathcal Q_\sigma(\Gamma)$; in other words, we have:
\begin{equation}\label{4.9}
h(F^1,F^2)=\sum (h_{(1)}(F^1),h_{(2)}(F^2))\qquad\forall\textrm{ }F^1,F^2\in \mathcal Q_\sigma(\Gamma), h\in \mathfrak h_1
\end{equation} where the coproduct $\Delta:\mathfrak h_1\longrightarrow 
\mathfrak h_1\otimes \mathfrak h_1$ is given by $\Delta(h)=\sum  h_{(1)}\otimes 
h_{(2)}$ for any $h\in \mathfrak h_1$. 
\end{thm}

\begin{proof} (a) We need to verify that for any $F\in \mathcal Q_\sigma(\Gamma)$ and 
any $\alpha\in GL_2^+(\mathbb Q)$, we have $([Y,X](F))_\alpha=X(F)_\alpha$. We know that for any element $g\in \mathcal{QM}$ and hence in particular for the element $F_\alpha\in \mathcal{QM}$,
we have $[Y,X](g)=X(g)$. The result now follows from the definition of the action of $X$
and $Y$ in \eqref{4.8}. 

\medskip
(b) The Lie action of $\mathfrak l_1$ on $\mathcal Q_\sigma(\Gamma)$ from part (a) induces an action
of the universal enveloping algebra $\mathfrak h_1$ on $\mathcal Q_\sigma(\Gamma)$. In order
to prove \eqref{4.9}, it suffices to prove the result for the generators $X$ and $Y$. We have:
\begin{equation}\label{4.10}
\begin{array}{l}
(X(F^1,F^2))_\alpha=X((F^1,F^2)_\alpha)\\
 = X\left( \underset{\beta\in \Gamma\backslash SL_2(\mathbb Z)}{\sum}F^1_{\beta\sigma}
 \cdot (F^2_{\alpha\sigma^{-1}\beta^{-1}}||\sigma\beta)\right)\\ 
 = \underset{\beta\in \Gamma\backslash SL_2(\mathbb Z)}{\sum} X(F^1_{\beta\sigma})
 \cdot (F^2_{\alpha\sigma^{-1}\beta^{-1}}||\sigma\beta)+ \underset{\beta\in \Gamma\backslash GL_2^+(\mathbb Q)}{\sum}F^1_{\beta\sigma}
 \cdot X(F^2_{\alpha\sigma^{-1}\beta^{-1}}||\sigma\beta)\\
  = \underset{\beta\in \Gamma\backslash SL_2(\mathbb Z)}{\sum} X(F^1_{\beta\sigma})
 \cdot (F^2_{\alpha\sigma^{-1}\beta^{-1}}||\sigma\beta)+ \underset{\beta\in \Gamma\backslash SL_2(\mathbb Z)}{\sum}F^1_{\beta\sigma}
 \cdot (X(F^2_{\alpha\sigma^{-1}\beta^{-1}})||\sigma\beta)\\
 = (X(F^1),F^2))_\alpha + (F^1,X(F^2))_\alpha\\
\end{array}
\end{equation} In \eqref{4.10}, we have used the fact that $\sigma\beta\in SL_2(\mathbb Z)$ and
hence $X(F^2_{\alpha\sigma^{-1}\beta^{-1}}||\sigma\beta)=X(F^2_{\alpha\sigma^{-1}\beta^{-1}})||\sigma\beta$. We can similarly verify the relation \eqref{4.9} for $Y\in \mathfrak h_1$. This proves the result. 

\end{proof}

\medskip
Our next aim is to show that $\mathcal Q_\sigma(\Gamma)$ is a right $\mathcal Q(\Gamma)$-module. Thereafter, we will consider
the Hopf algebra $\mathcal H_1$ defined in Section 3.2 and show that there is a ``Hopf action'' of   
$\mathcal H_1$ on the right $\mathcal Q(\Gamma)$-module $\mathcal Q_\sigma(\Gamma)$. 

\medskip
\begin{thm} Let $\sigma\in SL_2(\mathbb Z)$ and let $\Gamma=\Gamma(N)$ be a principal congruence
subgroup of $SL_2(\mathbb Z)$. Then, $\mathcal Q_\sigma(\Gamma)$ carries  a right 
$\mathcal Q(\Gamma)$-module structure defined by:
\begin{equation}\label{4.11}
(F^1\ast F^2)_\alpha :=\underset{\beta\in \Gamma\backslash GL_2^+(\mathbb Q)}{\sum}
F^1_{\beta\sigma}\cdot (F^2_{\alpha\sigma^{-1}\beta^{-1}}|\beta)
\end{equation} for any $F^1\in \mathcal Q_\sigma(\Gamma)$ and any
$F^2\in \mathcal Q(\Gamma)$. 

\end{thm}

\begin{proof} We take $\gamma\in \Gamma$. Then, since $F^1\in \mathcal Q_\sigma(\Gamma)$ and 
$F^2\in \mathcal Q(\Gamma)$, we have:
\begin{equation}
F^1_{\gamma\beta\sigma}=F^1_{\beta\sigma}\qquad
F^2_{\alpha\sigma^{-1}\beta^{-1}\gamma^{-1}}|\gamma\beta=F^2_{\alpha\sigma^{-1}\beta^{-1}}|\gamma^{-1}\gamma\beta=F^2_{\alpha\sigma^{-1}\beta^{-1}}|\beta
\end{equation} It follows that the sum in \eqref{4.11} is well defined, i.e., it does not depend
on the choice of coset representatives for $\Gamma$ in $GL_2^+(\mathbb Q)$. Further, it is clear
that $(F^1\ast F^2)_{\gamma\alpha}=(F^1\ast F^2)_\alpha$. In order to show that 
$F^1\ast F^2\in \mathcal Q_\sigma(\Gamma)$, it remains to show that
$(F^1\ast F^2)_{\alpha\gamma}=(F^1\ast F_2)_\alpha||\sigma\gamma\sigma^{-1}$. By definition,
we know that:
\begin{equation}\label{4.13}
\begin{array}{ll}
(F^1\ast F^2)_{\alpha\gamma} & = \underset{\beta\in \Gamma\backslash GL_2^+(\mathbb Q)}{\sum} 
F^1_{\beta\sigma}\cdot (F^2_{\alpha\gamma\sigma^{-1}\beta^{-1}}|\beta)\\
\end{array}
\end{equation} We now set $\delta=\beta\sigma\gamma^{-1}\sigma^{-1}$. This allows us to rewrite
\eqref{4.13} as follows:
\begin{equation}\label{4.14}
\begin{array}{ll}
(F^1\ast F^2)_{\alpha\gamma}& = \underset{\delta\in \Gamma\backslash GL_2^+(\mathbb Q)}{\sum}
F^1_{\delta\sigma\gamma}\cdot (F^2_{\alpha\sigma^{-1}\delta^{-1}}|\delta\sigma\gamma\sigma^{-1})\\
& = \underset{\delta\in \Gamma\backslash GL_2^+(\mathbb Q)}{\sum}
(F^1_{\delta\sigma}||\sigma\gamma\sigma^{-1})\cdot ((F^2_{\alpha\sigma^{-1}\delta^{-1}}|\delta)|\sigma\gamma\sigma^{-1}))\\
& =\left(\underset{\delta\in \Gamma\backslash GL_2^+(\mathbb Q)}{\sum}
F^1_{\delta\sigma}\cdot (F^2_{\alpha\sigma^{-1}\delta^{-1}}|\delta)\right) ||\sigma\gamma\sigma^{-1}\\
& =(F^1\ast F^2)_\alpha ||\sigma\gamma\sigma^{-1}\\
\end{array}
\end{equation} Hence, $(F^1\ast F^2)\in \mathcal Q_\sigma(\Gamma)$. In order to show that
$\mathcal Q_\sigma(\Gamma)$ is a right $\mathcal Q(\Gamma)$-module, we need to check that
$F^1\ast (F^2\ast F^3)=(F^1\ast F^2)\ast F^3$ for any $F^1\in \mathcal Q_\sigma(\Gamma)$ and
any $F^2,F^3\in \mathcal Q(\Gamma)$. For this, we note that:
\begin{equation}\label{4.15}
(F^1\ast F^2)_\alpha=\underset{\alpha_2\alpha_1=\alpha}
{\sum} F^1_{\alpha_1}\cdot (F^2_{\alpha_2}|\alpha_1\sigma^{-1})\qquad \forall\textrm{ }\alpha\in GL_2^+(\mathbb Q)
\end{equation} where the sum in \eqref{4.15} is taken over all pairs $(\alpha_1,\alpha_2)$ such that
$\alpha_2\alpha_1=\alpha$ modulo the the following equivalence relation:
\begin{equation}
(\alpha_1,\alpha_2)\sim (\gamma\alpha_1,\alpha_2\gamma^{-1})\qquad\forall\textrm{ }\gamma\in \Gamma
\end{equation} It follows that for any $\alpha\in GL_2^+(\mathbb Q)$, we have:
\begin{equation}\label{4.17}
((F^1\ast F^2)\ast F^3)_\alpha=\underset{\alpha_3\alpha_2\alpha_1=\alpha}{\sum}F^1_{\alpha_1}\cdot
(F^2_{\alpha_2}|\alpha_1\sigma^{-1})\cdot (F^3_{\alpha_3}|\alpha_2\alpha_1\sigma^{-1})
\end{equation} where the sum in \eqref{4.17} is taken over all triples $(\alpha_1,\alpha_2,\alpha_3)$ such
that $\alpha_3\alpha_2\alpha_1=\alpha$ modulo the following equivalence relation:
\begin{equation}\label{4.18}
(\alpha_1,\alpha_2,\alpha_3)\sim (\gamma\alpha_1,\gamma'\alpha_2\gamma^{-1},\alpha_3\gamma'^{-1})
\qquad\forall\textrm{ }\gamma,\gamma'\in \Gamma
\end{equation} On the other hand, we have:
\begin{equation}\label{4.19}
\begin{array}{ll}
(F^1\ast (F^2\ast F^3))_\alpha & =\underset{\alpha_2'\alpha_1=\alpha}{\sum} F^1_{\alpha_1}\cdot ((F^2\ast F^3)_{\alpha'_2}|\alpha_1\sigma^{-1})\\
& = \underset{\alpha_3\alpha_2\alpha_1=\alpha}{\sum} F^1_{\alpha_1}\cdot (F^2_{\alpha_2}|\alpha_1\sigma^{-1})\cdot (F^3_{\alpha_3}|\alpha_2\alpha_1\sigma^{-1})\\
\end{array}
\end{equation} Again, we see that the sum in \eqref{4.19} is taken over all triples $(\alpha_1,\alpha_2,\alpha_3)$ such
that $\alpha_3\alpha_2\alpha_1=\alpha$ modulo the  equivalence relation in \eqref{4.18}. From \eqref{4.17} and \eqref{4.19}, it follows that $(F^1\ast (F^2\ast F^3))_\alpha=((F^1\ast F^2)\ast F^3)_\alpha$. This proves the
result. 
\end{proof}

\medskip
We are now ready to describe the action of the Hopf algebra $\mathcal H_1$ on $\mathcal Q_\sigma(\Gamma)$. From Section 3.2, we know that 
$\mathcal H_1$ is generated by $X$, $Y$, $\{\delta_n\}_{n\geq 1}$ which satisfy the relations  \eqref{3.2.46}, \eqref{3.2.47qz}, \eqref{3.2.48qz}.

\medskip
\begin{thm} Let $\Gamma=\Gamma(N)$ be a principal congruence subgroup of $SL_2(\mathbb Z)$
and choose some $\sigma\in SL_2(\mathbb Z)$. 

\medskip
(a) The collection of $\sigma$-twisted quasimodular Hecke operators of level $\Gamma$ can be made into an $\mathcal H_1$-module as follows; for any $F\in \mathcal Q_\sigma(\Gamma)$ and $\alpha\in GL_2^+(\mathbb Q)$:
\begin{equation}\label{4.20}
X(F)_\alpha:=X(F_\alpha) \qquad Y(F)_\alpha:=Y(F_\alpha)\qquad \delta_n(F)_\alpha:=X^{n-1}(\mu_{\alpha\sigma^{-1}})\cdot F_\alpha\qquad\forall\textrm{ }n\geq 1
\end{equation}

\medskip
(b) The Hopf algebra $\mathcal H_1$ has a ``Hopf action'' on the right $\mathcal Q(\Gamma)$-module
$\mathcal Q_\sigma(\Gamma)$; in other words, for any $F^1\in \mathcal Q_\sigma(\Gamma)$ and
any $F^2\in \mathcal Q(\Gamma)$, we have:
\begin{equation}\label{4.21z}
h(F^1\ast F^2)=\sum h_{(1)}(F^1)\ast h_{(2)}(F^2)\qquad \forall\textrm{ }h\in \mathcal H_1
\end{equation} where the coproduct $\Delta:\mathcal H_1\longrightarrow \mathcal H_1\otimes 
\mathcal H_1$ is given by $\Delta(h)=\sum h_{(1)}\otimes h_{(2)}$ for each $h\in \mathcal H_1$. 
\end{thm}

\begin{proof} (a) For any $F\in \mathcal Q_\sigma(\Gamma)$, we have already checked in the proof 
of Proposition \ref{Prop4.3} that 
$X(F)$, $Y(F)\in \mathcal Q_\sigma(\Gamma)$. Further, from \eqref{3.2.57X}, we know that for any
$\alpha\in GL_2^+(\mathbb Q)$ and $\gamma\in \Gamma$, we have:
\begin{equation}\label{4.22}
\begin{array}{c}
\mu_{\gamma\alpha\sigma^{-1}}=\mu_{\gamma}|\alpha\sigma^{-1}+\mu_{\alpha\sigma^{-1}}=\mu_{\alpha\sigma^{-1}}\\
\mu_{\alpha\gamma\sigma^{-1}}=\mu_{\alpha\sigma^{-1}}|\sigma\gamma\sigma^{-1}+
\mu_{\sigma\gamma\sigma^{-1}}=\mu_{\alpha\sigma^{-1}}|\sigma\gamma\sigma^{-1}\\
\end{array}
\end{equation} Hence, for any $F\in \mathcal Q_\sigma(\Gamma)$, we have:
\begin{equation}
\begin{array}{c}
\delta_n(F)_{\gamma\alpha}=X^{n-1}(\mu_{\gamma\alpha\sigma^{-1}})\cdot F_{\gamma\alpha}=
X^{n-1}(\mu_{\alpha\sigma^{-1}})\cdot F_{\alpha}=\delta_n(F)_\alpha\\
\delta_n(F)_{\alpha\gamma}=X^{n-1}(\mu_{\alpha\gamma\sigma^{-1}})\cdot F_{\alpha\gamma} 
=X^{n-1}(\mu_{\alpha\sigma^{-1}}|\sigma\gamma\sigma^{-1})\cdot (F_\alpha||\sigma\gamma\sigma^{-1})
=\delta_n(F)_\alpha||\sigma\gamma\sigma^{-1}\\
\end{array}
\end{equation} Hence, $\delta_n(F)\in \mathcal Q_\sigma(\Gamma)$. In order to show that there is an action
of the Lie algebra $\mathcal L_1$ (and hence of its universal eneveloping algebra $\mathcal H_1$) on
$\mathcal Q_\sigma(\Gamma)$, it remains to check the commutator relations \eqref{3.2.46} between the operators
$X$, $Y$ and $\delta_n$ acting on $\mathcal Q_\sigma(\Gamma)$. We have already checked that 
$[Y,X]=X$ in the proof of Proposition \ref{Prop4.3}. Since $X$ is a derivation on $\mathcal{QM}$ and 
$\delta_n(F)_\alpha=X^{n-1}(\mu_{\alpha\sigma^{-1}})\cdot F_\alpha$, we have: 
\begin{equation}
\begin{array}{ll}
[X,\delta_n](F)_\alpha & = X(X^{n-1}(\mu_{\alpha\sigma^{-1}})\cdot F_\alpha) - X^{n-1}(\mu_{\alpha\sigma^{-1}})\cdot X(F_\alpha)\\
& = X(X^{n-1}(\mu_{\alpha\sigma^{-1}}))\cdot F_\alpha =X^n(\mu_{\alpha\sigma^{-1}})\cdot F_\alpha=\delta_{n+1}(F)_\alpha \\
\end{array}
\end{equation} Similarly, since $\mu_{\alpha\sigma^{-1}}
\in \mathcal M\subseteq \mathcal{QM}$ is of weight $2$  and $Y$ is a derivation on $\mathcal{QM}$, we have:
\begin{equation}
\begin{array}{ll}
[Y,\delta_n](F)_\alpha &=Y(X^{n-1}(\mu_{\alpha\sigma^{-1}})\cdot F_\alpha) - X^{n-1}(\mu_{\alpha\sigma^{-1}})\cdot Y(F_\alpha) \\
&=Y(X^{n-1}(\mu_{\alpha\sigma^{-1}}))\cdot F_\alpha=nX^{n-1}(\mu_{\alpha\sigma^{-1}})\cdot F_\alpha=n\delta_n(F)_\alpha \\
\end{array}
\end{equation} Finally, we can verify easily that $[\delta_k,\delta_l]=0$ for any $k$, $l\geq 1$.

\medskip
(b) In order to prove \eqref{4.21z}, it is enough to check this equality for the generators
$X$, $Y$ and $\delta_1\in \mathcal H_1$. For $F^1\in \mathcal Q_\sigma(\Gamma)$, $F^2\in \mathcal Q(\Gamma)$ and $\alpha\in GL_2^+(\mathbb Q)$, we have:
\begin{equation}\label{4.26rty}
\begin{array}{l}
(X(F^1\ast F^2))_\alpha=X((F^1\ast F^2)_\alpha)\\
= \underset{\beta\in \Gamma\backslash GL_2^+(\mathbb Q)}{\sum}X(F^1_{\beta\sigma}\cdot (F^2_{
\alpha\sigma^{-1}\beta^{-1}}|\beta)) \\
=\underset{\beta\in \Gamma\backslash GL_2^+(\mathbb Q)}{\sum}  X(F^1_{\beta\sigma})\cdot (F^2_{
\alpha\sigma^{-1}\beta^{-1}}|\beta) +  \underset{\beta\in \Gamma\backslash GL_2^+(\mathbb Q)}{\sum}  F^1_{\beta\sigma}\cdot X(F^2_{
\alpha\sigma^{-1}\beta^{-1}}|\beta) \\
=(X(F^1)\ast F^2)_\alpha + \underset{\beta\in \Gamma\backslash GL_2^+(\mathbb Q)}{\sum} F^1_{\beta\sigma}\cdot X(F^2_{
\alpha\sigma^{-1}\beta^{-1}})|\beta - \underset{\beta\in \Gamma\backslash GL_2^+(\mathbb Q)}{\sum}  F^1_{\beta\sigma}\cdot (\mu_{\beta^{-1}}|\beta)\cdot Y(F^2_{\alpha\sigma^{-1}\beta^{-1}})|\beta \\
=(X(F^1)\ast F^2)_\alpha + \underset{\beta\in \Gamma\backslash GL_2^+(\mathbb Q)}{\sum} F^1_{\beta\sigma}\cdot X(F^2_{
\alpha\sigma^{-1}\beta^{-1}})|\beta + \underset{\beta\in \Gamma\backslash GL_2^+(\mathbb Q)}{\sum}  F^1_{\beta\sigma}\cdot \mu_{\beta}\cdot Y(F^2_{\alpha\sigma^{-1}\beta^{-1}})|\beta \\
=(X(F^1)\ast F^2)_\alpha + (F^1\ast X(F^2))_\alpha + \sum_{\beta\in \Gamma\backslash GL_2^+(\mathbb Q)}
\delta_1(F)_{\beta\sigma}\cdot Y(F^2)_{\alpha\sigma^{-1}\beta^{-1}}|\beta \\
=(X(F^1)\ast F^2)_\alpha + (F^1\ast X(F^2))_\alpha + (\delta_1(F^1)\ast Y(F^2))_\alpha\\ 
\end{array}
\end{equation} In \eqref{4.26rty} above, we have used the fact that $0=\mu_{\beta^{-1}\beta}
=\mu_{\beta^{-1}}|\beta +\mu_\beta$. For $\alpha$, $\beta\in GL_2^+(\mathbb Q)$, it follows from \eqref{3.2.57X} that
\begin{equation}\label{216XXY}
\mu_{\alpha\sigma^{-1}}=\mu_{\alpha\sigma^{-1}\beta^{-1}\beta}=\mu_{\alpha\sigma^{-1}
\beta^{-1}}|\beta + \mu_{\beta} 
\end{equation} Since $F^2\in \mathcal Q(\Gamma)$ we know from \eqref{3.2.58} that  
$\delta_1(F^2)_{\alpha\sigma^{-1}\beta^{-1}}=
\mu_{\alpha\sigma^{-1}\beta^{-1}}\cdot F^2_{\alpha\sigma^{-1}\beta^{-1}}$.  Combining with \eqref{216XXY}, we have:
\begin{equation}
\begin{array}{l}
\delta_1((F^1\ast F^2))_\alpha=\mu_{\alpha\sigma^{-1}}\cdot (F^1\ast F^2)_\alpha= \underset{\beta\in \Gamma\backslash GL_2^+(\mathbb Q)}{\sum} \mu_{\alpha\sigma^{-1}}\cdot (F^1_{\beta\sigma}\cdot (F^2_{
\alpha\sigma^{-1}\beta^{-1}}|\beta))\\
=  \underset{\beta\in \Gamma\backslash GL_2^+(\mathbb Q)}{\sum} 
(\mu_\beta\cdot F^1_{\beta\sigma})\cdot (F^2_{\alpha\sigma^{-1}\beta^{-1}}|\beta)
+  \underset{\beta\in \Gamma\backslash GL_2^+(\mathbb Q)}{\sum}   F^1_{\beta\sigma} 
\cdot (\mu_{\alpha\sigma^{-1}\beta^{-1}}\cdot F^2_{\alpha\sigma^{-1}\beta^{-1}})|\beta\\
=(\delta_1(F^1)\ast F^2)_\alpha + (F^1\ast\delta_1(F^2))_\alpha
\end{array}
\end{equation} Finally, from the definition of $Y$, it is easy to show that  $(Y(F^1\ast F^2))_\alpha
=(Y(F^1)\ast F^2)_\alpha + (F^1\ast Y(F^2))_\alpha$. 

\end{proof}

\medskip

\medskip
\subsection{The operators $X_\tau:\mathcal Q_\sigma(\Gamma)\longrightarrow \mathcal Q_{\tau\sigma}(\Gamma)$ and
Hopf action}

\medskip

\medskip
Let $\Gamma=\Gamma(N)$ be a principal congruence subgroup and choose some $\sigma\in 
SL_2(\mathbb Z)$. In Section 4.1, we have only considered operators $X$, $Y$ and $\{\delta_n\}_{n\geq 1}$ that are endomorphisms of $\mathcal Q_\sigma(\Gamma)$. In this section, we will define an operator
\begin{equation}\label{4.2.28}
X_\tau:\mathcal Q_\sigma(\Gamma)\longrightarrow \mathcal Q_{\tau\sigma}(\Gamma)
\end{equation} for   $\tau\in SL_2(\mathbb Z)$. In particular, we consider the commuting family $\left\{\rho_n:=\begin{pmatrix} 
1 & n \\ 0 & 1 \\ \end{pmatrix}\right\}_{n\in \mathbb Z}$  of matrices in $SL_2(\mathbb Z)$ and write $\sigma(n):=
\rho_n\cdot \sigma$. Then, we have operators:
\begin{equation}\label{4.2.29}
X_{\rho_n}:\mathcal Q_{\sigma(m)}(\Gamma)\longrightarrow \mathcal Q_{\sigma(m+n)}(\Gamma) \qquad \forall\textrm{ }m,n\in \mathbb Z
\end{equation} acting ``between the levels'' of the graded module $\mathbb Q_\sigma(\Gamma):=
\underset{m\in \mathbb Z}{\bigoplus}\mathcal Q_{\sigma(m)}(\Gamma)$. We already know that $\mathcal Q_\sigma(\Gamma)$ carries an action of the Hopf algebra $\mathfrak h_1$. Further, $\mathfrak h_1$ has a Hopf action on the pairing on $\mathcal Q_\sigma(\Gamma)$ in the sense of Proposition \ref{Prop4.3}. We will now show that $\mathfrak h_1$ can be naturally embedded  into a larger Hopf algebra $\mathfrak h_{\mathbb Z}$ acting on 
$\mathbb Q_\sigma(\Gamma)$ that incorporates the operators $X_{\rho_n}$ in 
\eqref{4.2.29}. Finally, we will show that the pairing
on $\mathcal Q_\sigma(\Gamma)$ can be extended to a pairing:
\begin{equation}
(\_\_,\_\_):\mathcal Q_{\sigma(m)}(\Gamma)\otimes \mathcal Q_{\sigma(n)}(\Gamma)
\longrightarrow \mathcal Q_{\sigma(m+n)}(\Gamma)\qquad \forall\textrm{ }m,n\in \mathbb Z
\end{equation} This gives us a pairing on $\mathbb Q_\sigma(\Gamma)$ and we prove that this pairing carries a Hopf action of 
$\mathfrak h_{\mathbb Z}$. We start by defining the operators $X_\tau$ mentioned in \eqref{4.2.28}. 

\medskip
\begin{thm}\label{Prop4.6} (a) Let $\Gamma=\Gamma(N)$ be a principal congruence subgroup of $SL_2(\mathbb Z)$
and choose $\sigma\in SL_2(\mathbb Z)$. 

\medskip
(a) For each $\tau\in SL_2(\mathbb Z)$, we have a morphism:
\begin{equation}\label{4.2.31}
X_\tau:\mathcal Q_\sigma(\Gamma)\longrightarrow \mathcal Q_{\tau\sigma}(\Gamma)\qquad X_\tau(F)_\alpha
:=X(F_\alpha)||\tau^{-1}\qquad \forall\textrm{ }F\in \mathcal Q_\sigma(\Gamma),\textrm{ }\alpha\in GL_2^+(\mathbb Q)
\end{equation}

\medskip
(b) Let $\tau_1$, $\tau_2\in SL_2(\mathbb Z)$ be two matrices such that $\tau_1\tau_2=\tau_2\tau_1$. Then, the commutator $[X_{\tau_1},X_{\tau_2}]=0$. 

\end{thm}

\begin{proof} (a) We choose any $F\in \mathcal Q_\sigma(\Gamma)$. From \eqref{4.2.31}, it is clear that $X_\tau(F)_{\gamma\alpha}=X_\tau(F)_\alpha$ for
any $\gamma\in \Gamma$ and $\alpha\in GL_2^+(\mathbb Q)$. Further, we note that:
\begin{equation}\label{4.33xp}
\begin{array}{ll}
X_{\tau}(F)_{\alpha\gamma}=X (F_{\alpha\gamma})||\tau^{-1}&= X(F_\alpha ||\sigma\gamma\sigma^{-1})||\tau^{-1} \\
& =X(F_\alpha||\tau^{-1})||\tau\sigma\gamma\sigma^{-1}\tau^{-1} \\
& = X_\tau(F_\alpha)||((\tau\sigma)\gamma(\sigma^{-1}\tau^{-1}))\\
\end{array}
\end{equation} It follows from \eqref{4.33xp} that $X_\tau(F)\in \mathcal Q_{\tau\sigma}(\Gamma)$ for any $F\in \mathcal Q_\sigma(\Gamma)$. 

\medskip
(b) Since $\tau_1$ and $\tau_2$ commute, both $X_{\tau_1}X_{\tau_2}$ and 
$X_{\tau_2}X_{\tau_1}$ are operators from $\mathcal Q_{\sigma}(\Gamma)$
to $\mathcal Q_{\tau_1\tau_2\sigma}(\Gamma)=\mathcal Q_{\tau_2\tau_1\sigma}(\Gamma)$. For any
$F\in \mathcal Q_\sigma(\Gamma)$, we have ($\forall$ $\alpha\in GL_2^+(\mathbb Q)$):
\begin{equation}
(X_{\tau_1}X_{\tau_2}(F))_\alpha=X(X_{\tau_2}(F)_\alpha)||\tau_1^{-1}=X^2(F_\alpha)||\tau_2^{-1}\tau_1^{-1}
=X^2(F_\alpha)||\tau_1^{-1}\tau_2^{-1} = (X_{\tau_2}X_{\tau_1}(F))_\alpha
\end{equation} This proves the result. 
\end{proof}

\medskip
As mentioned before, we now consider the commuting family $\left\{\rho_n:=\begin{pmatrix} 
1 & n \\ 0 & 1 \\ \end{pmatrix}\right\}_{n\in \mathbb Z}$  of matrices in $SL_2(\mathbb Z)$ and set $\sigma(n):=
\rho_n\cdot \sigma$ for any $\sigma\in SL_2(\mathbb Z)$. We want to define a pairing on the graded module 
$\mathbb Q_\sigma(\Gamma)=\underset{m\in \mathbb Z}{\bigoplus}\mathcal Q_{\sigma(m)}(\Gamma)$ that extends the pairing on $\mathcal Q_\sigma(\Gamma)$. In fact, we will prove a more general result.

\medskip
\begin{thm} Let $\Gamma=\Gamma(N)$ be a principal congruence subgroup of $SL_2(\mathbb Z)$
and choose $\sigma\in SL_2(\mathbb Z)$. Let $\tau_1$, $\tau_2\in SL_2(\mathbb Z)$ be two matrices such that $\tau_1\tau_2=\tau_2\tau_1$.  Then, there exists a pairining 
\begin{equation}\label{4.34.2}
(\_\_,\_\_):\mathcal Q_{\tau_1\sigma}(\Gamma)\otimes \mathcal Q_{\tau_2\sigma}(\Gamma)
\longrightarrow \mathcal Q_{\tau_1\tau_2\sigma}(\Gamma)
\end{equation} defined as follows: for any $F^1\in \mathcal Q_{\tau_1\sigma}(\Gamma)$ and 
any $F^2\in \mathcal Q_{\tau_2\sigma}(\Gamma)$, we set:
\begin{equation}\label{4.2.35}
(F^1,F^2)_\alpha:=\underset{\beta\in \Gamma\backslash SL_2(\mathbb Z)}{\sum} (F^1_{\beta\sigma}||\tau_2^{-1})\cdot (F^2_{\alpha\sigma^{-1}\beta^{-1}}||\tau_2\sigma\beta\tau_1^{-1}\tau_2^{-1}) 
\qquad \forall\textrm{ }\alpha\in GL_2^+(\mathbb Q)
\end{equation} In particular, when $\tau_1=\tau_2=1$, the pairing in \eqref{4.2.35} reduces
to the  pairing on $\mathcal Q_\sigma(\Gamma)$  defined in \eqref{4.4}. 
\end{thm}

\begin{proof} We choose some $\gamma\in \Gamma$. Then, for any $\alpha\in GL_2^+(\mathbb Q)$,  $\beta\in SL_2(\mathbb Z)$,
we have $F^1_{\gamma\beta\sigma}=F^1_{\beta\sigma}$ and:
\begin{equation*}
 (F^2_{\alpha\sigma^{-1}\beta^{-1}
\gamma^{-1}}||\tau_2\sigma\gamma\beta\tau_1^{-1}\tau_2^{-1}) =(F^2_{\alpha\sigma^{-1}\beta^{-1}}||\tau_2\sigma\gamma^{-1}\sigma^{-1}\tau_2^{-1}\tau_2\sigma\gamma\beta 
\tau_1^{-1}\tau_2^{-1})= (F^2_{\alpha\sigma^{-1}\beta^{-1}}||\tau_2\sigma\beta 
\tau_1^{-1}\tau_2^{-1})
\end{equation*} It follows that the sum in \eqref{4.2.35} is well defined, i.e., independent of the choice
of coset representatives of $\Gamma$ in $SL_2(\mathbb Z)$. Additionally, we have:
\begin{equation}\label{4.2.36y}
(F^1,F^2)_{\alpha\gamma} :=\underset{\beta\in \Gamma\backslash SL_2(\mathbb Z)}{\sum} 
(F^1_{\beta\sigma}||\tau_2^{-1})\cdot (F^2_{\alpha\gamma\sigma^{-1}\beta^{-1}}||\tau_2
\sigma\beta 
\tau_1^{-1}\tau_2^{-1})
\end{equation} We now set  $\delta=\beta\sigma\gamma^{-1}
\sigma^{-1}$. Since $F^1\in \mathcal Q_{\tau_1\sigma}(\Gamma)$, 
we know that $F^1_{\delta\sigma\gamma}=F^1_{\delta\sigma}||\tau_1\sigma\gamma\sigma^{-1}\tau_1^{-1}$. Then, we can rewrite the expression in \eqref{4.2.36y} as follows:
\begin{equation}\label{4.2.37}
\begin{array}{ll}
(F^1,F^2)_{\alpha\gamma} &=\underset{\beta\in \Gamma\backslash SL_2(\mathbb Z)}{\sum} 
(F^1_{\delta\sigma\gamma}||\tau_2^{-1})\cdot (F^2_{\alpha\sigma^{-1}\delta^{-1}}||\tau_2\sigma\delta\sigma\gamma\sigma^{-1} 
\tau_1^{-1}\tau_2^{-1}) \\
&=\underset{\beta\in \Gamma\backslash SL_2(\mathbb Z)}{\sum} 
(F^1_{\delta\sigma}||\tau_1\sigma\gamma\sigma^{-1}\tau_1^{-1}\tau_2^{-1})\cdot (F^2_{\alpha\sigma^{-1}\delta^{-1}}||\tau_2\sigma\delta\sigma\gamma\sigma^{-1} 
\tau_1^{-1}\tau_2^{-1}) \\
 &=\left(\underset{\beta\in \Gamma\backslash SL_2(\mathbb Z)}{\sum} 
(F^1_{\delta\sigma}||\tau_2^{-1})\cdot (F^2_{\alpha\sigma^{-1}\delta^{-1}}||\tau_2\sigma\delta\tau_1^{-1}\tau_2^{-1})\right){||} \tau_1\tau_2\sigma\gamma\sigma^{-1} 
\tau_1^{-1}\tau_2^{-1}\\
&=(F^1,F^2)_\alpha || \tau_1\tau_2\sigma\gamma\sigma^{-1} 
\tau_1^{-1}\tau_2^{-1}\\
\end{array}
\end{equation} From \eqref{4.2.37} it follows that $(F^1,F^2)\in \mathcal Q_{\tau_1\tau_2\sigma}(\Gamma)$. 

\end{proof}

\medskip
In particular, it follows from the pairing in \eqref{4.34.2} that for any $m$, $n\in \mathbb Z$, we have a pairing
\begin{equation}\label{4.38.2}
(\_\_,\_\_):\mathcal Q_{\sigma(m)}(\Gamma)\otimes 
\mathcal Q_{\sigma(n)}(\Gamma)\longrightarrow \mathcal Q_{\sigma(m+n)}(\Gamma)
\end{equation} It is clear that \eqref{4.38.2} induces a pairing on $\mathbb Q_\sigma(\Gamma)=\underset{m\in \mathbb Z}{\bigoplus}\mathcal Q_{\sigma(m)}(\Gamma)$ for
each $\sigma\in SL_2(\mathbb Z)$. We will now define operators $\{X_n\}_{n\in \mathbb Z}$ and 
$Z$ on $\mathbb Q_\sigma(\Gamma)$. For each $n\in \mathbb Z$, the operator 
$X_n:\mathbb Q_\sigma(\Gamma)\longrightarrow \mathbb Q_\sigma(\Gamma)$ is induced
by the collection of operators:
\begin{equation}\label{4.2.39}
X_n^m:=X_{\rho_n}:\mathcal Q_{\sigma(m)}(\Gamma)\longrightarrow \mathcal Q_{\sigma(m+n)}(\Gamma)
\qquad \forall\textrm{ }m\in \mathbb Z
\end{equation} where, as mentioned before, $\rho_n=\begin{pmatrix} 1 & n \\ 0 & 1 \\ \end{pmatrix}
$. Then, $X_n:\mathbb Q_\sigma(\Gamma)\longrightarrow 
\mathbb  Q_\sigma(\Gamma)$ is an operator of homogeneous degree $n$ on the graded module
$\mathbb Q_\sigma(\Gamma)$. We also consider:
\begin{equation}\label{4.2.40}
Z:\mathcal Q_{\sigma(m)}(\Gamma)\longrightarrow \mathcal Q_{\sigma(m)}(\Gamma) 
\qquad Z(F)_\alpha:=mF_\alpha +Y(F_\alpha)\qquad \forall\textrm{ }F\in \mathcal Q_{\sigma(m)}(
\Gamma), \alpha\in GL_2^+(\mathbb Q)
\end{equation} This induces an operator $Z:\mathbb Q_\sigma(\Gamma)
\longrightarrow \mathbb Q_\sigma(\Gamma)$ of homogeneous degree $0$ on 
the graded module $\mathbb Q_\sigma(\Gamma)$. We will now show that $\mathbb Q_\sigma(\Gamma)$
is acted upon by a certain Lie algebra $\mathfrak l_{\mathbb Z}$ such that the Lie action incorporates
the operators $\{X_n\}_{n\in \mathbb Z}$ and $Z$ mentioned above. We define $\mathfrak l_{\mathbb Z}$
to be the Lie algebra with generators $\{Z,X_n|n\in \mathbb Z\}$ satisfying the following commutator
relations:
\begin{equation}\label{4.2.41}
[Z,X_n]=(n+1)X_n \qquad [X_n,X_{n'}]=0 \qquad \forall\textrm{ }n,n'\in \mathbb Z
\end{equation} In particular, we note that $[Z,X_0]=X_0$. It follows that the Lie algebra 
$\mathfrak l_{\mathbb Z}$ contains the Lie algebra $\mathfrak l_1$ defined in \eqref{3.59uu}. We now
describe the action of $\mathfrak l_{\mathbb Z}$ on $\mathbb Q_\sigma(\Gamma)$. 

\medskip
\begin{thm} Let $\Gamma=\Gamma(N)$ be a principal congruence subgroup of $SL_2(\mathbb Z)$
and let $\sigma\in SL_2(\mathbb Z)$. Then, the Lie algebra $\mathfrak l_{\mathbb Z}$ has a Lie action
on $\mathbb Q_\sigma(\Gamma)$.
\end{thm}

\begin{proof} We need to check that $[Z,X_n]=(n+1)X_n$ and $[X_n,X_{n'}]=0$, $\forall$ $n$, $n'\in \mathbb Z$
for the operators $\{Z, X_n | n\in\mathbb Z\}$ on $\mathbb Q_\sigma(\Gamma)$. From part (b) of
Proposition \ref{Prop4.6}, we know that $[X_n,X_{n'}]=0$. From \eqref{4.2.39} and \eqref{4.2.40}, it is clear that in order to show that $[Z,X_n]=(n+1)X_n$, we need to check that
$[Z,X_n^m]=(n+1)X_n^m:\mathcal Q_{\sigma(m)}(\Gamma)\longrightarrow 
\mathcal Q_{\sigma(m+n)}(\Gamma)$ for any given $m\in \mathbb Z$. For any $F\in \mathcal Q_{\sigma(m)}(\Gamma)$ and any $\alpha\in GL_2^+(\mathbb Q)$, we now check that:
\begin{equation}\label{4.2.42}
\begin{array}{c}
(ZX_n^m(F))_\alpha=(n+m)X_n^m(F)_\alpha+Y(X_n^m(F)_\alpha)=
(n+m)X(F_\alpha)||\rho_n^{-1}+YX(F_\alpha)||\rho_n^{-1}\\\
(X_n^mZ(F))_\alpha = X(Z(F)_\alpha)||\rho_n^{-1}= mX(F_\alpha)||\rho_n^{-1}+XY(F_\alpha)||\rho_n^{-1} \\
\end{array}
\end{equation} Combining \eqref{4.2.42} with the fact that $[Y,X]=X$, it follows that
$[Z,X_n^m]=(n+1)X_n^m$ for each $m\in \mathbb Z$. Hence, the result follows. 
\end{proof}

\medskip
We now consider the Hopf algebra $\mathfrak h_{\mathbb Z}$ that is the universal enveloping algebra
of the Lie algebra $\mathfrak l_{\mathbb Z}$. Accordingly, the coproduct $\Delta$ on $\mathfrak h_{\mathbb Z}$ 
is given by:
\begin{equation}\label{4.2.43}
\Delta(X_n)=X_n\otimes 1+1\otimes X_n \qquad \Delta(Z)=Z\otimes 1+1\otimes Z\qquad
\forall\textrm{ }n\in \mathbb Z
\end{equation} It is clear that $\mathfrak h_{\mathbb Z}$ contains the Hopf algebra $\mathfrak h_1$, the
universal enveloping algebra of $\mathfrak l_1$. From Proposition \ref{Prop4.3}, we know that
$\mathfrak h_1$ has a Hopf action on the pairing on $\mathcal Q_\sigma(\Gamma)$. We want to show
that $\mathfrak h_{\mathbb Z}$ has a Hopf action on the pairing on $\mathbb Q_\sigma(\Gamma)$. For this, we prove the following Lemma.

\medskip
\begin{lem}\label{lem49x} Let $\Gamma=\Gamma(N)$ be a principal congruence subgroup of $SL_2(\mathbb Z)$ and let $\sigma
\in SL_2(\mathbb Z)$. Let $\tau_1$, $\tau_2$,
$\tau_3\in SL_2(\mathbb Z)$ be three matrices such that $\tau_i\tau_j=\tau_j\tau_i$, 
$\forall$ $i$, $j\in \{1,2,3\}$. Then, for any  $F^1\in \mathcal Q_{\tau_1\sigma}
(\Gamma)$, $F^2\in \mathcal Q_{\tau_2\sigma}(\Gamma)$, we have:
\begin{equation}\label{4.2.44}
X_{\tau_3}(F^1,F^2)=(X_{\tau_3}(F^1),F^2)+ (F^1,X_{\tau_3}(F^2))
\end{equation}

\end{lem}

\begin{proof} Consider any $\alpha\in GL_2^+(\mathbb Q)$. Then, from the definition of
$X_{\tau_3}$, it follows that 
\begin{equation}\label{4.2.45}
\begin{array}{l}
X_{\tau_3}(F^1, F^2)_\alpha = \underset{\beta\in \Gamma\backslash SL_2(\mathbb Z)}{\sum} 
X((F^1_{\beta\sigma}||\tau_2^{-1})\cdot (F^2_{\alpha\sigma^{-1}\beta^{-1}}||\tau_2
\sigma\beta 
\tau_1^{-1}\tau_2^{-1}))||\tau_3^{-1} \\
=  \underset{\beta\in \Gamma\backslash SL_2(\mathbb Z)}{\sum} 
(X(F^1_{\beta\sigma})||\tau_2^{-1}\tau_3^{-1})\cdot  (F^2_{\alpha\sigma^{-1}\beta^{-1}}||\tau_2\sigma\beta 
\tau_1^{-1}\tau_2^{-1}\tau_3^{-1}) \\ \quad \quad + \underset{\beta\in \Gamma\backslash SL_2(\mathbb Z)}{\sum} 
(F^1_{\beta\sigma}||\tau_2^{-1}\tau_3^{-1})\cdot  (X(F^2_{\alpha\sigma^{-1}\beta^{-1}})||\tau_2\sigma\beta 
\tau_1^{-1}\tau_2^{-1}\tau_3^{-1})  \\
\end{array}
\end{equation} Since $F^1\in \mathcal Q_{\tau_1\sigma}(\Gamma)$, it follows that
$X_{\tau_3}(F^1)\in \mathcal Q_{\tau_1\tau_3\sigma}(\Gamma)$. Similarly, we see that 
$X_{\tau_3}(F^2)\in \mathcal Q_{\tau_2\tau_3\sigma}(\Gamma)$. It follows that: 
\begin{equation}\label{4.2.46}
\begin{array}{ll}
(X_{\tau_3}(F^1),F^2)_\alpha &=   \underset{\beta\in \Gamma\backslash SL_2(\mathbb Z)}{\sum} 
(X_{\tau_3}(F^1)_{\beta\sigma}||\tau_2^{-1})\cdot (F^2_{\alpha\sigma^{-1}\beta^{-1}}||\tau_2\sigma\beta 
\tau_1^{-1}\tau_2^{-1}\tau_3^{-1})\\ 
& =  \underset{\beta\in \Gamma\backslash SL_2(\mathbb Z)}{\sum} 
(X(F^1_{\beta\sigma})||\tau_2^{-1}\tau_3^{-1})\cdot (F^2_{\alpha\sigma^{-1}\beta^{-1}}||\tau_2\sigma\beta 
\tau_1^{-1}\tau_2^{-1}\tau_3^{-1})\\ 
(F^1,X_{\tau_3}(F^2))_\alpha & =  \underset{\beta\in \Gamma\backslash SL_2(\mathbb Z)}{\sum} 
(F^1_{\beta\sigma}||\tau_2^{-1}\tau_3^{-1})\cdot  (X_{\tau_3}(F^2)_{\alpha\sigma^{-1}\beta^{-1}}||\tau_2\tau_3\sigma\beta 
\tau_1^{-1}\tau_2^{-1}\tau_3^{-1})  \\
& =  \underset{\beta\in \Gamma\backslash SL_2(\mathbb Z)}{\sum} 
(F^1_{\beta\sigma}||\tau_2^{-1}\tau_3^{-1})\cdot  (X(F^2_{\alpha\sigma^{-1}\beta^{-1}})||\tau_3^{-1}\tau_2\tau_3\sigma\beta 
\tau_1^{-1}\tau_2^{-1}\tau_3^{-1})  \\
& =  \underset{\beta\in \Gamma\backslash SL_2(\mathbb Z)}{\sum} 
(F^1_{\beta\sigma}||\tau_2^{-1}\tau_3^{-1})\cdot  (X(F^2_{\alpha\sigma^{-1}\beta^{-1}})||\tau_2\sigma\beta 
\tau_1^{-1}\tau_2^{-1}\tau_3^{-1})  \\
\end{array} 
\end{equation} Comparing \eqref{4.2.45} and  \eqref{4.2.46}, the result of 
\eqref{4.2.44} follows. 

\end{proof}

\medskip
As a special case of Lemma \ref{lem49x}, it follows that for any
$F^1\in \mathcal Q_{\sigma(m)}(\Gamma)$ and $F^2\in \mathcal Q_{\sigma(m')}(\Gamma)$, 
we have:
\begin{equation}\label{4.2.47}
X_{\rho_n}(F^1,F^2)=X_n(F^1,F^2)=(X_n(F^1),F^2)+(F^1,X_n(F^2))\qquad \forall \textrm{ }n\in \mathbb Z
\end{equation}
We conclude by showing that $\mathfrak h_{\mathbb Z}$ has a Hopf action
on the pairing on $\mathbb Q_\sigma(\Gamma)$. 

\medskip
\begin{thm} Let $\Gamma=\Gamma(N)$ be a principal congruence subgroup of $SL_2(\mathbb Z)$ and let $\sigma
\in SL_2(\mathbb Z)$.  Then, the Hopf algebra 
$\mathfrak h_{\mathbb Z}$ has a Hopf action on the pairing on $\mathbb Q_\sigma(\Gamma)$. 
In other words, for $F^1$, $F^2\in \mathbb Q_{\sigma}(\Gamma)$, we have
\begin{equation}\label{4.2.48}
h(F^1,F^2)=\sum (h_{(1)}(F^1),h_{(2)}(F^2))
\end{equation} where the coproduct $\Delta:\mathfrak h_{\mathbb Z}
\longrightarrow \mathfrak h_{\mathbb Z}\otimes \mathfrak h_{\mathbb Z}$ is defined
by setting $\Delta(h):=\sum h_{(1)}\otimes h_{(2)}$ for each $h\in \mathfrak h_{\mathbb Z}$. 

\end{thm}

\begin{proof} It suffices to prove the result in the case where $F^1\in \mathcal Q_{\sigma(m)}(\Gamma)$, 
$F^2\in \mathcal Q_{\sigma(m')}(\Gamma)$ for some $m$, $m'\in \mathbb Z$. Further, it suffices
to prove the relation \eqref{4.2.48} for the generators  $\{Z,X_n|n\in \mathbb Z\}$ of the Hopf algebra
$\mathfrak{h}_{\mathbb Z}$. For the generators $X_n$, $n\in \mathbb Z$, this is already the result of
\eqref{4.2.47} which follows from Lemma \ref{lem49x}. Since  $\Delta(Z)=Z\otimes 1+1\otimes Z$, it remains
to show that \begin{equation}
Z(F^1,F^2)=(Z(F^1),F^2)+(F^1,Z(F^2)) \qquad \forall\textrm{ }F^1\in \mathcal Q_{\sigma(m)}(\Gamma),
F^2\in \mathcal Q_{\sigma(m')}(\Gamma)
\end{equation} By the definition of the pairing on $\mathbb Q_\sigma(\Gamma)$, we know that   $(F^1,F^2)\in 
\mathcal Q_{\sigma(m+m')}(\Gamma)$. Then, for any $\alpha\in GL_2^+(\mathbb Q)$, we have:
\begin{equation}
\begin{array}{ll}
Z(F^1,F^2)_\alpha & = (m+m')(F^1,F^2)_\alpha + Y(F^1,F^2)_\alpha \\
& =  (m+m') \underset{\beta\in \Gamma\backslash SL_2(\mathbb Z)}{\sum}  
((F^1_{\beta\sigma}||\rho_{m'}^{-1})\cdot (F^2_{\alpha\sigma^{-1}\beta^{-1}}||
\rho_{m'}\sigma\beta\rho_{m}^{-1}\rho_{m'}^{-1})) \\ & \quad +  \underset{\beta\in \Gamma\backslash SL_2(\mathbb Z)}{\sum}  
Y((F^1_{\beta\sigma}||\rho_{m'}^{-1})\cdot (F^2_{\alpha\sigma^{-1}\beta^{-1}}||
\rho_{m'}\sigma\beta\rho_{m}^{-1}\rho_{m'}^{-1}))\\
& = \underset{\beta\in \Gamma\backslash SL_2(\mathbb Z)}{\sum}   
((mF^1_{\beta\sigma}+Y(F^1_{\beta\sigma}))||\rho_{m'}^{-1})\cdot  (F^2_{\alpha\sigma^{-1}\beta^{-1}}||
\rho_{m'}\sigma\beta\rho_{m}^{-1}\rho_{m'}^{-1}) \\
&\quad + \underset{\beta\in \Gamma\backslash SL_2(\mathbb Z)}{\sum}   
(F^1_{\beta\sigma}||\rho_{m'}^{-1})\cdot ((m'F^2_{\alpha\sigma^{-1}\beta^{-1}}
+Y(F^2_{\alpha\sigma^{-1}\beta^{-1}}))||\rho_{m'}\sigma\beta\rho_{m}^{-1}\rho_{m'}^{-1})\\
& = \underset{\beta\in \Gamma\backslash SL_2(\mathbb Z)}{\sum}   
(Z(F^1)_{\beta\sigma}||\rho_{m'}^{-1})\cdot  (F^2_{\alpha\sigma^{-1}\beta^{-1}}||
\rho_{m'}\sigma\beta\rho_{m}^{-1}\rho_{m'}^{-1}) \\ 
&\quad + \underset{\beta\in \Gamma\backslash SL_2(\mathbb Z)}{\sum}   
(F^1_{\beta\sigma}||\rho_{m'}^{-1})\cdot (Z(F^2)_{\alpha\sigma^{-1}\beta^{-1}}||\rho_{m'}\sigma\beta\rho_{m}^{-1}\rho_{m'}^{-1})\\
&= (Z(F^1),F^2)_\alpha+(F^1,Z(F^2))_\alpha \\
\end{array} 
\end{equation}

\end{proof}

\end{document}